\documentclass[11pt,reqno]{amsart}
\usepackage{amsmath,amssymb,graphicx,mathrsfs,amsopn,times,stmaryrd,amsthm,color,bm}
\usepackage[colorlinks=true,citecolor=cyan,linkcolor=blue,pagebackref]{hyperref}

\usepackage{geometry}
\usepackage{pgf}
\usepackage{tikz}

\usetikzlibrary{arrows, decorations.pathmorphing, backgrounds, positioning, fit, petri, automata}
\tikzset{font={\fontsize{10pt}{12}\selectfont}}
\numberwithin{equation}{section}
\newtheorem{thm}{Theorem}[section]
\newtheorem{prop}[thm]{Proposition}
\newtheorem{lem}[thm]{Lemma}
\newtheorem{cor}[thm]{Corollary}

\theoremstyle{definition} 
\newtheorem{eg}[thm]{Example}

\theoremstyle{remark}
\newtheorem{rem}[thm]{Remark}

\newcommand{\beq}{\begin{equation}}
\newcommand{\eeq}{\end{equation}}
\newcommand{\be}{\begin{equation*}}
\newcommand{\ee}{\end{equation*}}
\newcommand{\bs}{\boldsymbol}

\newcommand{\C}{\mathbb{C}}

\newcommand{\Z}{\mathbb{Z}}

\newcommand{\bB}{\mathbb{B}}

\newcommand{\mc}{\mathcal}

\newcommand{\cR}{\mathcal{R}}
\newcommand{\cT}{\mathcal{T}}

\newcommand{\gl}{\mathfrak{gl}}
\newcommand{\h}{\mathfrak{h}}

\newcommand{\fkS}{\mathfrak{S}}
\newcommand{\fkT}{\mathfrak{T}}

\newcommand{\Wr}{\mathrm{Wr}}
\newcommand{\rY}{\mathrm{Y}}
\newcommand{\End}{\mathrm{End}}

\newcommand{\sing}{{\mathrm{sing}}}   
\newcommand{\ch}{{\mathrm{ch}}} 
\newcommand{\str}{{\mathrm{str}}} 

\newcommand{\pa}{\partial}

\newcommand{\gge}{\geqslant}
\newcommand{\lle}{\leqslant}
\newcommand{\la}{\lambda}

\newcommand{\bti}{\bm{t_i}}
\newcommand{\btij}{\bm{t_{ij}}}
\newcommand{\bals}{\bm{\alpha_s}}
\newcommand{\bbes}{\bm{\beta_s}}
\newcommand{\bla}{\bm\lambda}

\newcommand{\glMN}{\mathfrak{gl}_{m|n}}

\newcommand{\Uone}{\mathrm{U}(\mathfrak{gl}_{1|1})}
\newcommand{\YglMN}{\mathrm{Y}(\mathfrak{gl}_{m|n})}
\newcommand{\Yone}{\mathrm{Y}(\mathfrak{gl}_{1|1})}
\newcommand{\bmx}{\begin{pmatrix}}    
\newcommand{\emx}{\end{pmatrix}}

\newcommand{\qedd}{\tag*{$\square$}}
\newcommand{\vSi}{\varSigma}

\begin{document}
\pagestyle{myheadings}
\setcounter{page}{1}

\title[Supersymmetric XXX spin chains]{On the supersymmetric XXX spin chains associated to $\gl_{1|1}$}

\author{Kang Lu and Evgeny Mukhin}
\address{K.L.: Department of Mathematical Sciences,
Indiana University-Purdue University\newline
\strut\kern\parindent Indianapolis, 402 N.Blackford St., LD 270,
Indianapolis, IN 46202, USA}\email{lukang@iupui.edu}
\address{E.M.: Department of Mathematical Sciences,
Indiana University-Purdue University\newline
\strut\kern\parindent Indianapolis, 402 N.Blackford St., LD 270,
Indianapolis, IN 46202, USA}\email{emukhin@iupui.edu}

\begin{abstract} We study the $\mathfrak{gl}_{1|1}$ supersymmetric XXX spin chains. We give an explicit description of the algebra of Hamiltonians acting on any cyclic tensor products of polynomial evaluation $\mathfrak{gl}_{1|1}$ Yangian modules. It follows that there exists a bijection between common eigenvectors (up to proportionality) of the algebra of Hamiltonians and monic divisors of an explicit polynomial written in terms of the Drinfeld polynomials. In particular our result implies that each common eigenspace of the algebra of Hamiltonians has dimension one. We also give dimensions of the generalized eigenspaces. We show that when the tensor product is irreducible, then all eigenvectors can be constructed using Bethe ansatz. We express the transfer matrices associated to symmetrizers and anti-symmetrizers of vector representations in terms of the first transfer matrix and the center of the Yangian.
\medskip

\end{abstract}

\maketitle
\thispagestyle{empty}
\section{Introduction}	
Spin chains have been at the center of the study of the integrable models since the introduction of Heisenberg magnet by H. Bethe back in 1931. The literature on the subject is immense and keeps growing. While the even case is by far the most popular, it is now clear that the spin chains associated to Lie superalgebras are the integral part of the picture. 

The supersymmetric spin chains were introduced back to 1980s, see \cite{KS,K}. They have enjoyed a surge of interest in the recent years, see e.g. \cite{BR,HLM,HLPRS16,PRS,MR14,TZZ}. However, we know much less about these models compared to the even case. The presence of fermionic roots creates a number of new features which are not yet well understood.

This paper is devoted to the study of the supersymmetric spin chains associated to the super Yangian $\Yone$.
This case is remarkable as it is sufficiently simple on one hand and it is complex enough to have all the phenomena of supersymmetry on the other hand. So it provides a perfect testing ground for methods and conjectures. 

The simplicity of the $\Yone$ spin chain is apparent as the model can be written in terms of free fermions and the corresponding Bethe ansatz equations decouple. Unsurprisingly, for generic values of parameters, the Bethe ansatz method gives the complete information about the spectrum of the model. However, when the Bethe equation has roots of non-trivial multiplicity, the situation is more subtle as Hamiltonians develop Jordan blocks. Moreover, the algebra of Hamiltonians apriori is not finitely generated since all anti-symmetric powers of the superspace $\C^{1|1}$ are non-trivial and each power gives a non-zero transfer matrix. Finally, the standard geometric Langlands philosophy asks for a description of the eigenvectors of the Hamiltonians in terms of ``opers". The $\Yone$ ``oper" is expected to be a ratio of two order-one difference operators and have a universal nature.

We are able to clarify all these points. Let us discuss our findings in more detail.

\medskip 

We consider tensor products $L(\bla,\bm b)=\bigotimes_{s=1}^k L_{\la^{(s)}}(b_s)$ of polynomial evaluation modules of $\Yone$, where $\bla=(\la^{(1)},\dots,\la^{(k)})$ is a sequence of polynomial $\gl_{1|1}$-weights and $\bm b=(b_1,\dots,b_k)$ is a sequence of complex numbers. 

We have the action of the commutative algebra of transfer matrices corresponding to polynomial modules which we call the Bethe algebra, see Section \ref{sec ber}. The Bethe algebra commutes with the diagonal action of $\mathfrak{gl}_{1|1}$. It turns out that the image of the Bethe algebra in $\End(L(\bla,\bm b))$ 
is generated by the first transfer matrix $\str\, T(x)=T_{11}(x)-T_{22}(x)$, where $T(x)$ is the matrix generating $\Yone$. More precisely, we show
if one adds to $\Yone$  the inverse of the central element $T_{11}^{(1)} - T_{22}^{(1)}$, then
all transfer matrices corresponding to polynomial modules can be explicitly expressed in terms of the first transfer matrix and the quantum Berezinian (which is central in $\Yone$), see Theorem \ref{thm transfer relations}.

This result is understood as a relation in the Grothendieck ring of the category of finite-dimensional representations of $\Yone$. Namely, the composition factors of a certain tensor product of evaluation vector representations are all isomorphic to the evaluation of a given symmetric power of the vector representation, up to one-dimensional modules, see Example \ref{eg vect rep}. However, instead of discussing the universal R-matrix, we give a Bethe ansatz method proof.

\medskip

Then we need to find the spectrum of the first transfer matrix acting on the subspace $(L(\bla,\bm b))^\sing_{(n-l,l)}$ of singular vectors in $L(\bla,\bm b)$ of a weight $\sum_{i=1}^k \la^{(i)}-l\alpha=(n-l,l)$, $l=1,\dots,k-1$. We manage more than that: we give an explicit description  of the image of the Bethe algebra in the endomorphism ring of the subspace $(L(\bla,\bm b))^\sing_{(n-l,l)}$. Namely, we show that if $L(\bla,\bm b)$ is cyclic, then the image of the Bethe algebra in $\End((L(\bla,\bm b))^\sing_{(n-l,l)})$ has dimension ${k-1}\choose{l}$  and is isomorphic to the algebra
\beq\label{eq bethe alg}
\C[w_1,\dots,w_{k-1}]^{\fkS_{l}\times \fkS_{k-l-1}}/\langle n\prod_{i=1}^{k-1}(x-w_i)-\gamma_{\bla,\bm b}(x)\rangle,
\eeq
where polynomial $\gamma_{\bla,\bm b}(x)$ is constructed from $\bla, \bs b$:
$$
\gamma_{\bla,\bm b}(x)=\prod_{s=1}^k(x-b_s+\la_1^{(s)})-\prod_{s=1}^k(x-b_s-\la_2^{(s)}).
$$
Moreover, the space $(L(\bla,\bm b))^\sing_{(n-l,l)}$ is identified with the regular representation over this algebra as a module over the Bethe algebra, see Theorem \ref{thm tensor irr}.

It follows that the eigenvectors (up to proportionality) of the Bethe algebra in $(L(\bla,\bm b))^\sing_{(n-l,l)}$ are in a bijective correspondence with the monic 
polynomials of degree $l$ which divide $\gamma_{\bla,\bm b}(x)$.  In the generic situation, the roots of polynomial $\gamma_{\bla,\bm b}$ are simple. Therefore, it has exactly ${k-1}\choose {l}$ distinct monic divisors of degree $l$ and the corresponding Bethe vectors form a basis of the space $(L(\bla,\bm b))^\sing_{(n-l,l)}$. This is exactly the standard Bethe ansatz.
However if $\gamma_{\bla,\bm b}(x)$ has multiple roots, the number of monic divisors is smaller. And to each monic divisor $y$ we have exactly one eigenvector of the Bethe algebra and a generalized eigenspace of dimension given by the product of binomial coefficients
$$\prod_{a\in\C} {{\textrm{Mult}_a(\gamma_{\bla,\bm b}(x))}\choose{\textrm{Mult}_a(y(x))}},$$ 
where $\textrm{Mult}_a(f)$ denotes the order of zero of $f(x)$ at $x=a$.
Moreover, when the tensor product is irreducible, then the eigenvalues and eigenvectors of the Bethe algebra can be found by the usual formulas, see \eqref{eq bv gl11}, \eqref{eq gl11 eigenvalue in y} and Theorem \ref{thm completeness general}.

We prove this result by adopting a method of \cite{MTV09}. 

\medskip

Finally, let us discuss the $\Yone$ ``opers".  Given a monic divisor $y$ of $\gamma_{\bla,\bm b}(x)$, or, in other words, an eigenvector $v_y$ of the Bethe algebra, following \cite{HLM}, we have a rational difference operator $\mc {D}_y$, see \eqref{eq diff oper y}. From the explicit formula for the eigenvalue, one sees that the coefficients of the rational difference operator in this case are essentially eigenvalues of the first transfer matrix acting on $v_y$. Again, we improve on that as follows. By \cite{MR14}, the Berezinian  $\textrm{Ber}(1-T^t(x)e^{-\pa_x})$ is a generating function for the transfer matrices, see \eqref{eq diff trans}. We show that $\textrm{Ber}(1-T^t(x)e^{-\pa_x})v_y=\mathcal{D}_y v_y$ by a brute force (and somewhat tedious) computation. It follows that there is a universal formula for the rational difference operator in terms of the first transfer matrix and the quantum Berezinian, which produces $\mc D_y$ when applied to the vector $v_y$ for all $y$, see Corollary \ref{universal oper}.

\medskip

We also describe and prove similar results for the quasi-periodic case, where the monodromy matrix $T(x)$ is twisted by a diagonal invertible two-by-two matrix $Q$, so that the first transfer matrix has the form $q_1T_{11}(x)-q_2 T_{22}(x)$, see Section \ref{sec main thm q} and Theorem \ref{thm tensor irr q}.

\medskip

This paper deals with tensor products of polynomial modules which provide a natural supersymmetric analog of the finite-dimensional modules in the even case.
We expect that for $\gl_{1|1}$ spin chain in the case of arbitrary tensor products of finite-dimensional modules, the image of Bethe algebra still has the form \eqref{eq bethe alg}, except for the case $q_1=q_2=1$, $n=\sum_{s=1}^{k}(\la_1^{(s)}+\la_2^{(s)})= 0$. The exceptional case is even more interesting, since the tensor product is not semisimple as a $\gl_{1|1}$-module and some singular vectors are in the image of the creation operator $e_{21}$. Then the Bethe ansatz is expected to describe the eigenvectors of the transfer matrix which are not in that image, see \cite[Section 8.3]{HMVY}. The methods of this paper are not applicable for arbitrary tensor products and one needs a different approach. 

\medskip 

The paper is organized as follows. In Section \ref{sec notation}, we fix notations and discuss basic facts of the algebraic Bethe ansatz. Then we study the space $\mc V^{\fkS}$ and its properties in Section \ref{sec space VS}. Section \ref{sec main thms} contains the main theorems. Section \ref{sec proof} is dedicated to the proofs of main theorems. We discuss the higher transfer matrices and the relations between higher transfer matrices and the Bethe algebra in Section \ref{sec higher}.

\medskip 

{\bf Acknowledgments.} We thank V. Tarasov for interesting discussions. This work was partially supported by a grant from the Simons Foundation \#353831.

\section{Notation}\label{sec notation}
\subsection{Lie superalgebra $\gl_{1|1}$ and its representations}
A \emph{vector superspace} $V = V_{\bar 0}\oplus V_{\bar 1}$ is a $\Z_2$-graded vector space. Elements of $V_{\bar 0}$ are called \emph{even}; elements of
$V_{\bar 1}$ are called \emph{odd}. We write $|v|\in\{\bar 0,\bar 1\}$ for the parity of a homogeneous element $v\in V$. Set $(-1)^{\bar 0}=1$ and $(-1)^{\bar 1}=-1$.

Consider the vector superspace $\C^{1|1}$, where $\dim(\C^{1|1}_{\bar 0})=1$ and  $\dim(\C^{1|1}_{\bar 1})=1$. We choose a homogeneous basis $v_1,v_2$ of $\C^{1|1}$ such that $|v_1|=\bar 0$ and $|v_2|=\bar 1$. For brevity we shall write their
parities as $|v_i|=|i|$. Denote by $E_{ij}\in\End(\C^{1|1})$ the linear operator of parity $|i|+|j|$ such that $E_{ij}v_r=\delta_{jr}v_i$ for $i,j,r=1,2$. 

The Lie superalgebra $\gl_{1|1}$ is spanned by elements $e_{ij}$, $i,j=1, 2$, with parities $e_{ij}=|i|+|j|$ and the supercommutator relations are given by
\[
[e_{ij},e_{rs}]=\delta_{jr}e_{is}-(-1)^{(|i|+|j|)(|r|+|s|)}\delta_{is}e_{rj}.
\]Let $\h$ be the commutative Lie subalgebra of $\gl_{1|1}$ spanned by $e_{11},e_{22}$. Denote the universal enveloping algebras of $\gl_{1|1}$ and $\h$ by $\Uone$ and $\mathrm{U}(\h)$, respectively.

We call a pair $\la=(\la_1,\la_2)$ of complex numbers a $\gl_{1|1}$-\emph{weight}. A $\gl_{1|1}$-weight $\la$ is \textit{non-degenerate} if $\la_1+\la_2\ne 0$.

Let $M$ be a $\gl_{1|1}$-module. A non-zero vector $v\in M$ is called \emph{singular} if $e_{12}v=0$. Denote the subspace of all singular vectors of $M$ by $(M)^{\rm sing}$.  A non-zero vector $v\in M$ is called \emph{of weight} $\la=(\la_1,\la_2)$ if $e_{11}v=\la_1 v$ and $e_{22}v=\la_2 v$. Denote by $(M)_\la$ the subspace of $M$ spanned by vectors of weight $\la$. Set $(M)^{\sing}_\la=(M)^{\sing}\cap (M)_\la$.

Let $\bla=(\la^{(1)},\dots,\la^{(k)})$ be a sequence of $\gl_{1|1}$-weights, we denote by $|\bla|$ the sum
\[
|\bla|=\sum_{s=1}^k(\la^{(s)}_1+\la^{(s)}_2).
\]
Denote by $L_\la$ the irreducible $\gl_{1|1}$-module generated by an even singular vector $v_\la$ of weight $\la$. Then $L_{\la}$ is two-dimensional if $\la$ is non-degenerate and one-dimensional otherwise. Clearly, $\C^{1|1}\cong L_{\omega_1}$, where $\omega_1=(1,0)$. 

A $\gl_{1|1}$-module $M$ is called a \emph{polynomial module} if $M$ is a submodule of $(\C^{1|1})^{\otimes n}$ for some $n\in \Z_{\gge 0}$. We say that $\la$ is a \emph{polynomial weight} if $L_\la$ is a polynomial module. Weight $\la=(\la_1,\la_2)$ is a polynomial weight if and only if $\la_1,\la_2\in \Z_{\gge 0}$ and either $\la_1>0$ or $\la_1=\la_2=0$. We also write $L_{(\la_1,\la_2)}$ for $L_{\la}$.

For non-degenerate polynomial weights $\la=(\la_1,\la_2)$ and $\mu=(\mu_1,\mu_2)$, we have
\[
L_{(\la_1,\la_2)}\otimes L_{(\mu_1,\mu_2)}=L_{(\la_1+\mu_1,\la_2+\mu_2)}\oplus L_{(\la_1+\mu_1-1,\la_2+\mu_2+1)}.
\]

Define a \emph{supertrace} $\str:\End(\C^{1|1})\to \C$, which is supercyclic,
\[
\str(E_{ij})=(-1)^{|j|}\delta_{ij},\qquad \str([E_{ij},E_{rs}])=0,
\]
where $[\cdot,\cdot]$ is the supercommutator.

Define the supertranspose $t$,
\[
t:\End(\C^{1|1})\to \End(\C^{1|1}),\quad E_{ij}\mapsto (-1)^{|i||j|+|i|}E_{ji}.
\]
The supertranspose is an anti-homomorphism and respects the supertrace,
\beq\label{eq transpose cyclic}
(AB)^t=(-1)^{|A||B|}B^tA^t,\qquad \str(A)=\str(A^t), 
\eeq
for all $2\times 2$ matrices $A$, $B$.

\subsection{Current superalgebra $\gl_{1|1}[t]$}
Denote by $\gl_{1|1}[t]$ the Lie superalgebra $\gl_{1|1}\otimes\C[t]$ of $\gl_{1|1}$-valued polynomials with the point-wise supercommutator. Call $\gl_{1|1}[t]$ the \emph{current superalgebra} of $\gl_{1|1}$. We identify $\gl_{1|1}$ with the subalgebra $\gl_{1|1}\otimes 1$ of constant polynomials in $\gl_{1|1}[t]$.

We write $e_{ij}[r]$ for $e_{ij}\otimes t^r$, $r\in \Z_{\gge 0}$. A basis of $\gl_{1|1}[t]$ is given by $e_{ij}[r]$, $i,j=1,2$ and $r\in \Z_{\gge 0}$. They satisfy the supercommutator relations
\[
[e_{ij}[r],e_{kl}[s]]=\delta_{jk}e_{il}[r+s]-(-1)^{(|i|+|j|)(|k|+|l|)}\delta_{il}e_{kj}[r+s].
\]In particular, one has $(e_{12}[r])^2=(e_{21}[r])^2=0$ and $e_{21}[r]e_{21}[s]=-e_{21}[s]e_{21}[r]$ in the universal enveloping superalgebra $\mathrm{U}(\gl_{1|1}[t])$. 

The universal enveloping superalgebra $\mathrm{U}(\gl_{1|1}[t])$ is a Hopf superalgebra with the coproduct given by
\[
\Delta(X)=X\otimes 1+1\otimes X,\quad \text{for }\ X\in \gl_{1|1}[t].
\]
There is a natural $\Z_{\gge 0}$-gradation on $\mathrm{U}(\gl_{1|1}[t])$ such that $\deg(e_{ij}[r])=r$.

Let $M$ be a $\Z_{\gge 0}$-graded space with finite-dimensional homogeneous components. Let $M_j\subset M$ be the homogeneous component of degree $j$. We call the formal power series in variable $q$,
\beq\label{eq grade ch}
\ch(M)=\sum_{j=0}^\infty \dim(M_j)\,q^j
\eeq
the \emph{graded character} of $M$.

\subsection{Super Yangian $\Yone$}\label{sec yangian}
We recall the standard facts about super Yangian $\Yone$ and fix notation, see e.g. \cite{Naz}.
Let $P\in\End(\C^{1|1})\otimes \End(\C^{1|1})$ be the graded flip operator given by
\[
P=\sum_{i,j=1}^{2}(-1)^{|j|}E_{ij}\otimes E_{ji}.
\]
For two homogeneous vectors $v_1,v_2$ in $\C^{1|1}$, we have $P(v_1\otimes v_2)=(-1)^{|v_1||v_2|}v_2\otimes v_1$. Define the rational R-matrix $\cR(x)\in \End(\C^{1|1})\otimes \End(\C^{1|1})$ by $\cR(x)=1+P/x$. The R-matrix $\cR(x)$ satisfies the Yang-Baxter equation
\beq\label{eq yangbaxter}
\cR^{(1,2)}(x_1-x_2)\cR^{(1,3)}(x_1)\cR^{(2,3)}(x_2)=\cR^{(2,3)}(x_2)\cR^{(1,3)}(x_1)\cR^{(1,2)}(x_1-x_2).
\eeq

The \emph{super Yangian} $\Yone$ is a unital associative superalgebra generated by the generators $T_{ij}^{(r)}$ of parity $|i|+|j|$, $i,j=1,2$ and $r\gge 1$, with the defining relation
\beq\label{eq RTT}
\cR^{(1,2)}(x_1-x_2) T^{(1,3)}(x_1)T^{(2,3)}(x_2)=T^{(2,3)}(x_2)T^{(1,3)}(x_1)\cR^{(1,2)}(x_1-x_2),
\eeq
where $T(x)\in \End(\C^{1|1})\otimes \Yone[[x^{-1}]]$ is the \emph{monodromy matrix}
\[
T(x)=\sum_{i,j=1}^{2}E_{ij}\otimes T_{ij}(x),\qquad T_{ij}(x)=\sum_{r=0}^{\infty}T_{ij}^{(r)}x^{-r},\qquad T_{ij}^{(0)}=\delta_{ij}.
\]
Note that the monodromy matrix is even.

The defining relation \eqref{eq RTT} gives
\beq\label{com relations}
\begin{split}
(x_1-x_2)[T_{ij}(x_1),T_{rs}(x_2)]=&\ (-1)^{|i||r|+|s||i|+|s||r|}\big(T_{rj}(x_2)T_{is}(x_1)-T_{rj}(x_1)T_{is}(x_2)\big)\\
=&\ (-1)^{|i||j|+|s||i|+|s||j|}\big(T_{is}(x_1)T_{rj}(x_2)-T_{is}(x_2)T_{rj}(x_1)\big).
\end{split}
\eeq
Equivalently,
\beq\label{eq com coord}
[T_{ij}^{(a)},T_{rs}^{(b)}]=(-1)^{|i||j|+|s||i|+|s||j|}\sum_{\ell=1}^{\min(a,b)}\left(T_{is}^{(\ell-1)}T_{rj}^{(a+b-\ell)}-T_{is}^{(a+b-\ell)}T_{rj}^{(\ell-1)}\right).
\eeq

The super Yangian $\Yone$ is a Hopf superalgebra with a coproduct and an  opposite coproduct given by 
\beq\label{eq cop}
\Delta: T_{ij}(x)\mapsto \sum_{r=1}^{2} T_{rj}(x)\otimes T_{ir}(x),\quad \Delta(T_{ij}^{(s)})=\sum_{r=1}^{2}\sum_{a=0}^s T_{rj}^{(a)}\otimes T_{ir}^{(s-a)},\qquad i,j=1,2,
\eeq
\[
\Delta^{\rm op}: T_{ij}(x)\mapsto \sum_{r=1}^{2}(-1)^{(|i|+|r|)(|r|+|j|)} T_{ir}(x)\otimes T_{rj}(x),\qquad i,j=1,2.
\]
The super Yangian $\Yone$ contains $\Uone$ as a Hopf subalgebra with the embedding given by $e_{ij}\mapsto (-1)^{|i|}T_{ji}^{(1)}$. By \eqref{com relations}, one has
\beq\label{zero mode com relations}
[T_{ij}^{(1)},T_{rs}(x)]=(-1)^{|i||r|+|s||i|+|s||r|}\big(\delta_{is}T_{rj}(x)-\delta_{rj}T_{is}(x)\big).
\eeq
The relation \eqref{zero mode com relations} implies that for any $r,s=1,2$,
\beq\label{eq com gll}
[E_{rs}\otimes 1+1\otimes e_{rs},T(x)]=0.
\eeq

Define the degree function on $\Yone$ by $\deg(T_{ij}^{(r)})=r-1$, then the super Yangian $\Yone$ is a filtered algebra. Let $\mathscr F_s\Yone$ be the subspace spanned by elements of degree $\lle s$, one has the increasing filtration $\mathscr F_0\Yone\subset \mathscr F_1\Yone\subset\cdots\subset \Yone$. 

The corresponding graded algebra $\mathrm{gr}\Yone$ inherits from $\Yone$ the Hopf structure. It is known $\mathrm{gr}\Yone$ is isomorphic to the universal enveloping superalgebra $\mathrm{U}(\gl_{1|1}[t])$. For instance, the graded product comes from \eqref{eq com coord} while the graded coproduct comes from \eqref{eq cop}.

For any complex number $z\in \C$, there is an automorphism
\[
\zeta_z:\Yone\to \Yone,\qquad T_{ij}(x)\to T_{ij}(x-z),
\]
where $(x-z)^{-1}$ is expanded as a power series in $x^{-1}$. The \emph{evaluation homomorphism} $\mathrm{ev}:\Yone\to \Uone$ is defined by the rule: $$T_{ji}^{(r)}\mapsto (-1)^{|i|}\delta_{1r}e_{ij},$$ for $r\in\Z_{>0}$.

For any $\gl_{1|1}$-module $M$ denote by $M(z)$ the $\Yone$-module obtained by pulling back of $M$ through the homomorphism $\mathrm{ev}\circ\zeta_z$. The module $M(z)$ is called an \emph{evaluation module with the evaluation point} $z$.

The map $\varpi_{\xi}:T(x)\to \xi(x)T(x)$, where
$\xi(x)$ is any formal power series in $x^{-1}$ with the leading term 1,
\beq\label{eq lead 1}
\xi(x)=1+\xi_1x^{-1}+\xi_2x^{-2}+\cdots\in\C[[x^{-1}]],
\eeq
defines an automorphism of $\Yone$. 

There is a one-dimensional module $\C v$ given by $T_{11}(x)v=T_{22}(x)v=\xi(x)v$, $T_{12}(x)v=T_{21}(x)v=0$, where $v$ is a homogeneous vector of parity $\bar i$.
We denote this module by $\C_{\bar i,\xi}$.

Let $\bm a=(a_1,\dots,a_n)$ be a sequence of complex numbers. Let $\bm b=(b_1,\dots,b_k)$ be another sequence of complex numbers, $\bla=(\la^{(1)},\dots,\la^{(k)})$ a sequence of $\gl_{1|1}$-weights. Set $L(\bla,\bm b)=\bigotimes_{s=1}^k L_{\la^{(s)}}(b_s)$, $V(\bm a)=\bigotimes_{i=1}^n\C^{1|1}(a_i)$, and
\beq\label{eq phi psi}
\varphi_{\bla,\bs b}(x)=\prod_{s=1}^k(x-b_s+\la_1^{(s)}),\qquad \psi_{\bla,\bm b}(x)=\prod_{s=1}^k(x-b_s-\la_2^{(s)}).
\eeq
When we write $L(\bla,\bm b)$, we shall always assume the participating $\gl_{1|1}$-weights $\la^{(s)}$ are non-degenerate. We also denote by $|0\rangle$ the vector $v_{\la^{(1)}}\otimes\dots\otimes v_{\la^{(k)}}\in L(\bla,\bm b)$. We call $|0\rangle$ a \emph{vacuum vector}. We call the $\Yone$-module $L(\bla,\bm b)$ \emph{cyclic} if it is generated by $|0\rangle$.

It is known from \cite{Z} that up to twisting by an automorphism $\varpi_{\xi}$ with proper $\xi(x)\in \C[[x^{-1}]]$ as in \eqref{eq lead 1}, every finite-dimensional irreducible representation of $\Yone$ is isomorphic to a tensor product of evaluation $\Yone$-modules. In other words, every finite-dimensional irreducible representation of $\Yone$ is of form $L(\bla,\bm b)\otimes \C_{\bar i,\xi}$. 

The following explicit description of irreducible and cyclic $\Yone$-modules in term of the highest weights
is similar to the one in the quantum affine case, see \cite{Zh}.
\begin{lem}\label{lem cyclic weyl}
The $\Yone$-module $L(\bla,\bm b)$ is irreducible if and only if $\varphi_{\bla,\bs b}(x)$ and $\psi_{\bla,\bs b}(x)$ are relatively prime. 

The $\Yone$-module $L(\bla,\bm b)$ is cyclic if and only if $b_j\ne b_i+\la^{(i)}_2+\la^{(j)}_1$ for $1\lle i< j\lle k$.
\end{lem}
\begin{proof}
The first statement follows from \cite[Theorems 4 and 5]{Z}. Proof of the second statement is similar to the proof of \cite[Theorem 4.2]{Zh}.
\end{proof}

We give decomposition of some tensor products of evaluation vector representations which are important for Section \ref{sec higher}.

\begin{eg}\label{eg vect rep}
We have the following equality in the Grothendieck ring,
\beq\label{eq tensor 2}
L_{n\omega_1}(z)\otimes L_{\omega_1}(z-n)=L_{(n+1)\omega_1}(z)+\big(\C_{\bar 1,\xi_n}\otimes L_{(n+1)\omega_1}(z)\big),
\eeq
where $\xi_n(x)=(x-z+n-1)/(x-z+n)$. Moreover, $\big(\C_{\bar 1,\xi_n}\otimes L_{(n+1)\omega_1}(z)\big)$ is the unique irreducible $\Yone$-submodule of $L_{n\omega_1}(z)\otimes L_{\omega_1}(z-n)$ while $L_{(n+1)\omega_1}(z)$ is the simple quotient module. 

Inductively using \eqref{eq tensor 2}, we have the equality in the Grothendieck ring,
\beq\label{eq tensor n}
\bigotimes_{i=1}^n L_{\omega_1}(z-i+1)= \sum_{\ell=0}^{n-1}\sum_{1\lle i_1<\dots<i_{\ell}\lle n-1}\C_{\bar \ell,\xi_{i_1}\cdots \xi_{i_{\ell}}}\otimes L_{n\omega_1}(z).
\eeq
Note that $\C_{\overline{n-1},\xi_{1}\cdots \xi_{n-1}}\otimes L_{n\omega_1}(z)$ is the unique irreducible $\Yone$-submodule of $\bigotimes_{i=1}^n L_{\omega_1}(z-i+1)$.\qed
\end{eg}

\subsection{Shapovalov form}
We write $v_1^{(s)}$ for $v_{\la^{(s)}}$. Let $v_2^{(s)}=e_{21}v_1^{(s)}$. Then $v_1^{(s)}$, $v_2^{(s)}$ is a basis of $L_{\la^{(s)}}$. 

Let $E_{ij}$, $i,j=1,2$, be the linear operator in $\End(L_{\la^{(s)}})$ of parity $|i|+|j|$ such that $E_{ij}v_{r}^{(s)}=\delta_{jr}v_i^{(s)}$ for $r=1,2$.

The R-matrix $R(x)\in \End(L_{\la^{(i)}}(b+x))\otimes \End(L_{\la^{(j)}}(b))$ is given by
\begin{align}
R(x)=&\ E_{11}\otimes E_{11}-\frac{\la^{(i)}_1+\la^{(j)}_2-x}{\la^{(j)}_1+\la^{(i)}_2+x}E_{22}\otimes E_{22}+\frac{\la^{(j)}_1-\la^{(i)}_1+x}{\la^{(j)}_1+\la^{(i)}_2+x}E_{11}\otimes E_{22} \nonumber\\
& +\frac{\la^{(i)}_2-\la^{(j)}_2+x}{\la^{(j)}_1+\la^{(i)}_2+x}E_{22}\otimes E_{11}-\frac{\la^{(i)}_1+\la^{(i)}_2}{\la^{(j)}_1+\la^{(i)}_2+x}E_{12}\otimes E_{21}+\frac{\la^{(j)}_1+\la^{(j)}_2}{\la^{(j)}_1+\la^{(i)}_2+x}E_{21}\otimes E_{12}.\label{eq R matrix}
\end{align}

For $(\la^{(i)}_1,\la^{(i)}_2)=(\la^{(j)}_1,\la^{(j)}_2)=(1,0)$, the $R$-matrix is a scalar multiple of the one used in the definition of the super Yangian $\Yone$ in Section \ref{sec yangian}:  $R(x)=\mathcal R(x) x/(1+x)=(x+P)/(1+x)$.

The R-matrix satisfies
\beq\label{R-matrix comult}
\Delta^{\rm op}(X)R(b_i-b_j)=R(b_i-b_j)\Delta(X),
\eeq
for all $X\in\Yone$. The module $L_{\la^{(i)}}(b_i)\otimes L_{\la^{(j)}}(b_j)$ is irreducible if and only if $R(b_i-b_j)$ is well-defined and invertible.

Define an anti-automorphism $\iota:\rY(\gl_{1|1})\to \rY(\gl_{1|1})$ by the rule, \beq\label{eq anti-in}
\iota(T_{ij}(x))=(-1)^{|i||j|+|i|}T_{ji}(x),\qquad i,j=1,2.
\eeq
One has $\iota(X_1X_2)=(-1)^{|X_1||X_2|}\iota(X_2)\iota(X_1)$ for $X_1,X_2\in \mathrm Y(\gl_{1|1})$. Note that for all $X\in \Yone$ we also have
\beq\label{iota comult}
\Delta \circ \iota (X)=(\iota\otimes \iota)\circ \Delta^{\rm op}(X).
\eeq

The \emph{Shapovalov form} $B_{\la^{(s)}}$ on $L_{\la^{(s)}}$ is a unique symmetric bilinear form such that $$B_{\la^{(s)}}(e_{ij}w_1,w_2)=(-1)^{(|i|+|j|)|w_1|}B_{\la^{(s)}}(w_1,(-1)^{|i||j|+|i|}e_{ji}w_2),$$ for all $i,j$ and $w_1,w_2\in L_{\la^{(i)}}$,  and $B_{\la^{(s)}}(v^{(s)}_1,v^{(s)}_1)=1$. Explicitly, it is given by
\[
B_{\la^{(s)}}(v^{(s)}_1,v^{(s)}_1)=1,\quad B_{\la^{(s)}}(v^{(s)}_1,v^{(s)}_2)=B_{\la^{(s)}}(v^{(s)}_2,v^{(s)}_1)=0,\quad B_{\la^{(s)}}(v^{(s)}_2,v^{(s)}_2)=-(\la^{(i)}_1+\la^{(i)}_2).
\]
The Shapovalov forms $B_{\la^{(s)}}$ on $L_{\la^{(s)}}$ induce a bilinear form $B_{\bla}=\bigotimes_{s=1}^kB_{\la^{(s)}}$ (respecting the usual sign convention) on $L(\bla)=\bigotimes_{s=1}^k L_{\la^{(s)}}$.

Let $R_{\bla,\bm b}\in \End(L(\bla))$ be the product of R-matrices,
\[
R_{\bla,\bm b}=\mathop{\overrightarrow\prod}\limits_{1\lle i\lle k}\,\mathop{\overrightarrow\prod}\limits_{i<j\lle k}R^{(i,j)}(b_i-b_j).
\]
Define a bilinear form $B_{\bla,\bm b}$ on $L(\bla,\bm b)$ by
\[
B_{\bla,\bm b}(w_1,w_2)=B_{\bla}(w_1,R_{\bla,\bm b}w_2),
\]
for all $w_1,w_2\in L(\bla,\bm b)$.

The bilinear form $B_{\bla,\bm b}$ satisfies
\[
B_{\bla,\bm b}(|0\rangle,|0\rangle)=1,\qquad B_{\bla,\bm b}(Xw_1,w_2)=(-1)^{|X||w_1|}B_{\bla,\bm b}(w_1,\iota(X)w_2), 
\]
for all $X\in \rY(\gl_{1|1})$, $w_1,w_2\in L(\bla,\bm b)$, see \eqref{R-matrix comult}, \eqref{iota comult}. Note that if $L(\bla,\bm b)$ is irreducible, then $B_{\bla,\bm b}$ is well-defined and non-degenerate.

\subsection{Bethe ansatz}\label{sec BA}
Let $Q=\mathrm{diag}(q_1,q_2)$ be an invertible diagonal matrix, where $q_1,q_2\in\C^\times$. Define the (\textit{twisted}) \emph{transfer matrix} $\cT_Q(x)$ to be the supertrace of $QT(x)$,
\[
\cT_Q(x)=\mathrm{str}(QT(x))=q_1T_{11}(x)-q_2T_{22}(x)=q_1-q_2+\sum_{r=1}^\infty(q_1T_{11}^{(r)}-q_2T_{22}^{(r)})x^{-r}\in\Yone[[x^{-1}]].
\]
Note that  the transfer matrix $\mc T_Q(x)$ essentially depends on the ratio $q_1/q_2$. 

We write simply $\mc T(x)$ for $\mc T_I(x)$, where $I$ is the identity matrix.

The twisted transfer matrices commute, $[\cT_Q(x_1),\cT_Q(x_2)]=0$. Moreover, $\cT_Q(x)$ commutes with $\mathrm U(\h)$ if $q_1\ne q_2$, and $\mathrm{U}(\gl_{1|1})$ if $q_1=q_2$.

Let $\bm b=(b_1,\dots,b_k)$ be a sequence of complex numbers, $\bla=(\la^{(1)},\dots,\la^{(k)})$ a sequence of $\gl_{1|1}$-weights. We are interested in finding eigenvalues and eigenvectors  of $\cT_Q(x)$ in $L(\bla,\bm b)$. To be more precise, we call \beq\label{eq series}
f(x)=q_1-q_2+\sum_{r=1}^\infty f_rx^{-r},\quad f_r\in \C,
\eeq 
an \emph{eigenvalue} of $\mc T_Q(x)$ if there exists a non-zero vector $v\in L(\bla,\bm b)$ such that $(q_1T_{11}^{(r)}-q_2T_{22}^{(r)})v=f_r v$ for all $r\in\Z_{>0}$. If $f(x)$ is a rational function, we consider it as a power series in $x^{-1}$ as \eqref{eq series}. The vector $v$ is called an \textit{eigenvector} of $\mc T_Q(x)$ corresponding to eigenvalue $f(x)$. We also define the \emph{eigenspace of $\mc T_Q(x)$ in $L(\bla,\bm b)$ corresponding to eigenvalue $f(x)$} as $\bigcap_{r=1}^\infty \ker((q_1T_{11}^{(r)}-q_2T_{22}^{(r)})|_{L(\bla,\bm b)}-f_r)$.

\medskip

It is sufficient to consider $L(\bla,\bm b)$ with $\la^{(i)}_2=0$ for all $i$. Indeed, if $L(\bla,\bm b)$ is an arbitrary tensor product and 
\[
\xi(x)=\prod_{s=1}^k\frac{x-b_s}{x-b_s-\la_2^{(s)}},
\]
then 
$$
L(\bla,\bm b)\otimes \C_{\bar{0},\xi}= L(\tilde \bla,\tilde {\bm b}),
$$
where
\[
\tilde{\la}^{(s)}=({\la}^{(s)}_1+{\la}^{(s)}_2,0),\quad \tilde b_s=b_s+\la_2^{(s)},\quad s=1,\dots,k.
\]
Identify $L(\bla,\bm b)\otimes \C_{\bar{0},\xi}$ with $L(\bla,\bm b)$ as vector spaces. 
Then $\cT_Q(x)$ acting on $L(\bla,\bm b)\otimes \C_{\bar{0},\xi}$ coincides with $\xi(x)\cT_Q(x)$ acting on  $L(\bla,\bm b)$ and therefore the problem of diagonalization of the transfer matrix in $L(\bla,\bm b)$ is reduced to diagonalization of the transfer matrix in $L(\tilde \bla,\tilde{\bm b})$.

\medskip

The main method to find eigenvalues and eigenvectors of $\cT_Q(x)$ in $L(\bla,\bm b)$ is the algebraic Bethe ansatz. Here we recall it from \cite{K,BR}.

Fix a non-negative integer $l$. Let $\bm t=(t_1,\dots,t_l)$ be a sequence of complex numbers. Define the polynomial $y_{\bm t}=\prod_{i=1}^l(x-t_i)$. We say that polynomial $y_{\bm t}$ \emph{represents} $\bm t$. 

A sequence of complex numbers $\bm t$ is called a \emph{solution to the Bethe ansatz equation associated to} $\bla$, $\bm b$, $l$ if
\beq\label{eq gl11 BAE}
y_{\bm t} \text{ divides the polynomial } q_1\varphi_{\bla,\bm b}(x)-q_2\psi_{\bla,\bm b}(x),
\eeq see \eqref{eq phi psi}. We do not distinguish solutions which differ by a permutation of coordinates (that is represented by the same polynomial).

\begin{rem}
In the literature, see e.g. \cite{BR}, one often simply calls $\bm t$ a solution of the Bethe ansatz equation if its coordinates satisfy the following system of algebraic equations:
\beq\label{eq gl11 BAE rem}
\frac{q_1}{q_2}\prod_{s=1}^k\frac{t_j-b_s+\la_1^{(s)}}{t_j-b_s-\la_2^{(s)}}=1, \qquad j=1,\dots,l.
\eeq
Note that \eqref{eq gl11 BAE rem} involves a single $t_j$ only, therefore $\bm t$ is a solution of the Bethe ansatz equation if and only if all $t_j$ are roots of the polynomial $q_1\varphi_{\bla,\bm b}(x)-q_2\psi_{\bla,\bm b}(x)$.  In a generic situation, $q_1\varphi_{\bla,\bm b}(x)-q_2\psi_{\bla,\bm b}(x)$ has no multiple roots, and it is sufficient to consider $\bs t$ with distinct coordinates. However, in general we need to allow the same number appear in $\bm t$ several times.
 According to our definition, a root $t_0$ of $q_1\varphi_{\bla,\bm b}(x)-q_2\psi_{\bla,\bm b}(x)$, occurs in $\bm t$ at most as many times as the order of $x-t_0$ in $q_1\varphi_{\bla,\bm b}(x)-q_2\psi_{\bla,\bm b}(x)$, see \cite[Section 3.2]{HLM}.    \qed
\end{rem}

Let $\la^{(\infty)}$ be the $\gl_{1|1}$-weight given by
\[
\la^{(\infty)}_1=\sum_{s=1}^k\la_1^{(s)}-l,\qquad \la^{(\infty)}_2=\sum_{s=1}^k\la_2^{(s)}+l.
\]
Define the \emph{off-shell Bethe vector} $\bB_l(\bm t)\in (L(\bla,\bm b))_{\la^{(\infty)}}$ by
\beq\label{eq bv gl11}
\bB_l(\bm t)=\prod_{i=1}^l\prod_{s=1}^k \frac{t_i-b_s}{t_i-b_s+\la_1^{(s)}}\prod\limits_{1\lle i<j\lle l}\frac{1}{t_j-t_i+1}\,T_{12}(t_1)\cdots T_{12}(t_l)\,|0\rangle.
\eeq
By \eqref{com relations}, $\bB_l(\bm t)$ is symmetric in $t_1,\dots,t_l$. We renormalize $\bB_l(\bm t)$ as follows,
\[
\widehat{\bB}_l(\bm t)=\Big(\prod_{i=1}^l\prod_{s=1}^k(t_i-b_s+\la_1^{(s)})\Big)\bB_l(\bm t).
\]
Then the Bethe vector $\widehat{\bB}_l(\bm t)$ is well-defined for all $\bm b$, $\bm t$.

If $\bm t$ is a solution of the Bethe ansatz equation \eqref{eq gl11 BAE}, we call $\bB_l(\bm t)$ (also $\widehat{\bB}_l(\bm t)$) an \emph{on-shell Bethe vector}.

Let $\bm t$ be a solution of the Bethe ansatz equation associated to $\bla$, $\bm b$, $l$.  The following theorem is well-known, see e.g. \cite{K,BR}.
\begin{thm}\label{thm gl11 eigenvalue in y}
If the on-shell Bethe vector $\widehat{\bB}_l(\bm t)$ is non-zero, then $\widehat{\bB}_l(\bm t)$ is an eigenvector of the transfer matrix $\cT_Q(x)$ with the corresponding eigenvalue 
\beq\label{eq gl11 eigenvalue in y}
\mc E^Q_{y_{\bm t},\bla,\bm b}(x)=\frac{y_{\bm t}(x-1)}{y_{\bm t}(x)}(q_1\varphi_{\bla,\bm b}(x)-q_2\psi_{\bla,\bm b}(x))\prod_{s=1 }^k(x-b_s)^{-1},
\eeq
where $\varphi_{\bla,\bm b}(x)$ and $\psi_{\bla,\bm b}(x)$ are given by \eqref{eq phi psi}.\qed
\end{thm}

Different on-shell Bethe vectors are orthogonal with respect to the bilinear form $B_{\bla,\bm b}(\cdot,\cdot)$.
\begin{cor}
Let $\bm{t_1}$ and $\bm{t_2}$ be two different solutions of the Bethe ansatz equation, then \[B_{\bla,\bm b}(\widehat{\bB}_{l_1}(\bm t_1),\widehat{\bB}_{l_2}(\bm t_2))=0.\]
\end{cor}
\begin{proof}
It follows from the equality 
\[
B_{\bla,\bm b}(\cT_Q(x)\widehat{\bB}_{l_1}(\bm t_1),\widehat{\bB}_{l_2}(\bm t_2))=B_{\bla,\bm b}(\widehat{\bB}_{l_1}(\bm t_1),\iota(\cT_Q(x))\widehat{\bB}_{l_2}(\bm t_2))=B_{\bla,\bm b}(\widehat{\bB}_{l_1}(\bm t_1),\cT_Q(x)\widehat{\bB}_{l_2}(\bm t_2))
\]and Theorem \ref{thm gl11 eigenvalue in y} since $\cT_Q(x)$ has different eigenvalues corresponding to $\widehat{\bB}_{l_1}(\bm t_1)$ and $\widehat{\bB}_{l_2}(\bm t_2)$.
\end{proof}

\begin{prop}[\cite{HLM}]
If $q_1=q_2$, then the on-shell Bethe vector $\widehat{\bB}_l(\bm t)$ is singular.
\end{prop}

It is important to know if the on-shell Bethe vectors are non-zero. The following theorem is a particular case of \cite[Theorem 4.1]{PRS} which asserts that the square of the norm of the on-shell Bethe vector is essentially given by the Jacobian of the Bethe ansatz equation. 

Let $\theta_1,\theta_2$ be differentiable functions in $x$. Denote by $\Wr(\theta_1,\theta_2)$ the \emph{Wronskian} of $\theta_1$ and $\theta_2$,
\[
\Wr(\theta_1,\theta_2)=\theta_1\theta_2'-\theta_1'\theta_2.
\]
\begin{thm}[\cite{PRS}]\label{thm norm BV}
The square of the norm of the on-shell Bethe vector $\widehat{\bB}_l(\bm t)$ is given by
\begin{align*}
B_{\bla,\bm b}(\widehat{\bB}_l(\bm t),\widehat{\bB}_l(\bm t))=&\ \Big(\frac{q_2}{q_1}\Big)^l\prod_{i=1}^l\frac{\Wr(\varphi_{\bla,\bm b},\psi_{\bla,\bm b})(t_i)}{y_{\bm t}'(t_i)}.
\end{align*}
\end{thm}
\begin{proof}
The statement for the case of $q_1=q_2$ is proved in \cite{PRS}. For the case of $q_1\ne q_2$, the proof is similar. One only needs to change the Bethe ansatz equation used in the proof of \cite[Lemma 7.1]{PRS} to the twisted case and alter the Korepin criteria (iii) and (iv) with appropriate multiple.
\end{proof}

\begin{rem}
Theorems \ref{thm gl11 eigenvalue in y} and \ref{thm norm BV} were proved for the case that $t_i\ne t_j$ for $i\ne j$. However, they still hold with the definition \eqref{eq gl11 BAE} which can be seen by analytic continuation, c.f. \cite[Theorem 3.2]{T}.
\end{rem}

\subsection{Completeness of Bethe ansatz}\label{sec complete}
We say that the Bethe ansatz is \emph{complete} if the following conditions are sastified,
\begin{enumerate}
    \item[(1)] on-shell Bethe vectors $\widehat{\bB}_l(\bm t)\ne 0$ for all $\bm t$ such that $y_{\bm t}$ divides $q_1\varphi_{\bla,\bm b}-q_2\psi_{\bla,\bm b}$ are non-zero;
    \item[(2)] all eigenvectors of the transfer matrix $\mc T_Q(x)$ in $L(\bla,\bm b)$ if $q_1\ne q_2$ and in $(L(\bla,\bm b))^\sing$ if $q_1=q_2$ are of the form $c\,\widehat{\bB}_l(\bm t)$ where $c\in\C^\times$ and $y_{\bm t}$ divides $q_1\varphi_{\bla,\bm b}-q_2\psi_{\bla,\bm b}$.
\end{enumerate}

Bethe vectors are obtained from the action of the Yangian on the vacuum vector $|0\rangle$. Therefore we restrict our interest to the case when $L(\bla,\bm b)$ is cyclic. The cyclic modules are described in Lemma \ref{lem cyclic weyl}. Note that $\dim L(\bla,\bm b)=2^k$ and $\dim (L(\bla,\bm b))^\sing=2^{k-1}$. For generic $\bm b$, the polynomial $q_1\varphi_{\bla,\bm b}(x)-q_2\psi_{\bla,\bm b}(x)$ has no multiple roots and hence has the desired number of distinct monic divisors. Thus the algebraic Bethe ansatz works well for generic situation. The following theorem is a minor generalization of \cite[Theorem A.6]{HLM}.  

\begin{thm}\label{thm complete}
Suppose $L(\bla,\bm b)$ is an irreducible $\Yone$-module. In the case of $q_1=q_2$ we assume
in addition that $|\bla|\ne 0$. If $q_1\varphi_{\bla,\bm b}-q_2\psi_{\bla,\bm b}$ has no multiple roots, then the transfer matrix $\mc T_Q(x)$ is diagonalizable and the Bethe ansatz is complete. In particular, for any given $\bla$ and generic $\bm b$, the transfer matrix $\mc T_Q(x)$ is diagonalizable and the Bethe ansatz is complete. 
\end{thm}
\begin{proof}
The proof is similar to that of \cite[Theorem A.6]{HLM}. Here we only show that the on-shell Bethe vectors $\widehat{\bB}_l(\bm t)$ are non-zero by using Theorem \ref{thm norm BV}. Note that
\[
B_{\bla,\bm z}(\widehat{\bB}_l(\bm t),\widehat{\bB}_l(\bm t))=\frac{q_2^l}{q_1^{2l}}\ \prod_{i=1}^l \frac{\Wr(q_1\varphi_{\bla,\bm b}-q_2\psi_{\bla,\bm b},\psi_{\bla,\bm b})(t_i)}{y'(t_i)}.
\]
It suffices to show $\Wr(q_1\varphi_{\bla,\bm b}-q_2\psi_{\bla,\bm b},\psi_{\bla,\bm b})(t_i)\ne 0$. Since $L(\bla,\bm b)$ is an irreducible $\Yone$-module, $\varphi_{\bla,\bm b}$ and $\psi_{\bla,\bm b}$ are relatively prime. So are $q_1\varphi_{\bla,\bm b}-q_2\psi_{\bla,\bm b}$ and $\psi_{\bla,\bm b}$. Note that $\big(q_1\varphi_{\bla,\bm b}-q_2\psi_{\bla,\bm b}\big)(t_i)=0$ and $q_1\varphi_{\bla,\bm b}-q_2\psi_{\bla,\bm b}$ has no multiple roots, we have $\big(q_1\varphi_{\bla,\bm b}-q_2\psi_{\bla,\bm b}\big)'(t_i)\ne 0$ and $\psi_{\bla,\bm b}(t_i)\ne 0$.
Hence $\Wr(q_1\varphi_{\bla,\bm b}-q_2\psi_{\bla,\bm b},\psi_{\bla,\bm b})(t_i)\ne 0$.

The second statement follows from the fact that the discriminant of $q_1\varphi_{\bla,\bm b}-q_2\psi_{\bla,\bm b}$ is a non-zero polynomial in $\bm b=(b_1,\dots,b_k)$. It is not hard to prove this fact by considering the leading coefficient of $b_1$ and using induction on $k$.
\end{proof}

Now we study what happens if polynomial $q_1\varphi_{\bla,\bm b}-q_2\psi_{\bla,\bm b}$ has multiple roots. 

\begin{lem}\label{lem nonzero}
If $L(\bla,\bm b)$ is an irreducible $\Yone$-module, then all on-shell Bethe vectors are non-zero.
\end{lem}
\begin{proof}
Let $\bm t$ be the solution of the Bethe ansatz equation represented by the polynomial $q_1\varphi_{\bla,\bm b}(x)-q_2\psi_{\bla,\bm b}(x)$ itself  (the largest monic divisor). By Theorem \ref{thm norm BV} and comparing the order of zeros and poles, one shows that the norm of the on-shell Bethe vector $\widehat{\bB}_{k}(\bm t)$ ($\widehat{\bB}_{k-1}(\bm t)$ if $q_1=q_2$) is non-zero, c.f. the proof of Theorem \ref{thm complete}. Since $\widehat{\bB}_{k}(\bm t)$ is obtained from all other on-shell Bethe vectors by applying a sequence of $T_{12}(x_0)$ for proper $x_0$'s with some scalar, all on-shell Bethe vectors are also non-zero.
\end{proof}

Since different on-shell Bethe vectors correspond to different eigenvalues of $\mc T_Q(x)$, see Theorem \ref{thm gl11 eigenvalue in y}, they are linearly independent. To show the completeness of Bethe ansatz, it suffices to show that all eigenvalues of $\cT_Q(x)$ have the form \eqref{eq gl11 eigenvalue in y} with a monic divisor $y_{\bm t}$ of $q_1\varphi_{\bla,\bm b}(x)-q_2\psi_{\bla,\bm b}(x)$ and that the corresponding eigenspaces have dimension one. 

Our first main theorem asserts that the Bethe ansatz is complete for irreducible tensor products of polynomial evaluation modules. 
\begin{thm}\label{thm completeness general}
Let $\bla$ be a sequence of polynomial $\gl_{1|1}$-weights. If $L(\bla,\bm b)$ is an irreducible $\Yone$-module, then the Bethe ansatz is complete.
\end{thm}

We will prove Theorem \ref{thm completeness general} in Section \ref{sec third}.

\section{Space $\mc{V}^\fkS$}\label{sec space VS}
In this section, we discuss the super-analog of $\mc{V}^+$ in \cite[Section 2]{GRTV}, c.f. \cite[Section 2]{MTV14}. Fix $n\in\Z_{>0}$. 

The symmetric group $\fkS_n$ acts naturally on $\C[z_1,\dots,z_n]$ by permuting variables. We call it the {\it standard} action of $\fkS_n$ on $\C[z_1,\dots,z_n]$. Denote by $\sigma_i(\bm z)$ the $i$-th elementary symmetric polynomial in $z_1,\dots,z_n$. The algebra of symmetric polynomials $\C[z_1,\dots,z_n]^\fkS$ is freely generated by $\sigma_1(\bm z),\dots,\sigma_n(\bm z)$. 

Fix $\ell\in \{0,1,\dots,n\}$. We have a subgroup $\fkS_{\ell}\times \fkS_{n-\ell}\subset \fkS_n$.
Then $\fkS_{\ell}$ permutes the first $\ell$ variables while $\fkS_{n-\ell}$ permutes the last $n-\ell$ variables. Denote by $\C[z_1,\dots,z_n]^{\fkS_{\ell}\times \fkS_{n-\ell}}$ the subalgebra of $\C[z_1,\dots,z_n]$ consisting of $\fkS_{\ell}\times \fkS_{n-\ell}$-invariant polynomials. It is known that $\C[z_1,\dots,z_n]^{\fkS_{\ell}\times \fkS_{n-\ell}}$ is a free $\C[z_1,\dots,z_n]^\fkS$-module of rank $n\choose{\ell}$.

\subsection{Definition of $\mc{V}^\fkS$}
Let $V=(\C^{1|1})^{\otimes n}$ be the tensor power of the vector representation of $\gl_{1|1}$. The $\gl_{1|1}$-module $V$ has weight decomposition
\[
V=\bigoplus_{\ell=0}^{n}(V)_{(n-\ell,\ell)}.
\]
The space $V$ has a basis $v_{i_1}\otimes \dots\otimes v_{i_n}$, where $i_j\in \{1,2\}$. Define $I_1=\{j\,|\,i_j=1\}$ and $I_2=\{j\,|\,i_j=2\}$. Then $I=(I_1,I_2)$ gives a two-partition of the set $\{1,2,\dots,n\}$. We simply write $v_I$ for the vector $v_{i_1}\otimes \dots\otimes v_{i_n}$. Denote by $\mc{I}_{\ell}$ the set of all two-partitions $I$ of $\{1,2,\dots,n\}$ such that $|I_{2}|=\ell$. Then the set of vectors $\{v_{I}~|~I\in \mc{I}_\ell\}$ forms a basis of $(V)_{(n-\ell,\ell)}$.

Let $\mc V$ be the space of polynomials in variables $\bm z=(z_1,z_2,\dots,z_n)$ with coefficients in $V$,
\[
\mc V=V\otimes \C[z_1,z_2,\dots,z_n].
\]The space $V$ is identified with the subspace $V\otimes 1$ of constant polynomials in $\mc V$. The space $\mc V$ has a natural grading induced from the grading on $\C[z_1,\dots,z_n]$ with $\deg(z_i)=1$. Namely, the degree of an element $v\otimes p$ in $\mc V$ is given by the degree of the polynomial $p$, $\deg(v\otimes p)=\deg\,p$. Denote by $\mathscr F_s\mc V$ the subspace spanned by all elements of degree $\lle s$. One has the increasing filtration $\mathscr F_0\mc V\subset\mathscr F_1\mc V\subset \dots\subset \mc V$. Clearly, the space $\End(\mc V)$ has a filtration structure induced from that on $\mc V$.

Let $P^{(i,j)}$ be the graded flip operator which acts on the $i$-th and $j$-th factors of $V$. Let $s_1,s_2,\dots,s_{n-1}$ be the simple permutations of the symmetric group $\fkS_n$. Define the \emph{modified} $\fkS_n$-action on $\mc V$ by the rule:
\begin{align}
\hat{s}_i:\bm f(z_1,\dots,z_n)\mapsto P^{(i,i+1)}&\bm f(z_1,\dots,z_{i+1},z_i,\dots,z_n)\nonumber\\+&\frac{\bm f(z_1,\dots,z_n)-\bm f(z_1,\dots,z_{i+1},z_i,\dots,z_n)}{z_i-z_{i+1}},\label{eq sn action}
\end{align}
for $\bm f(z_1,\dots,z_n)\in \mc V$. Note that the modified $\fkS_n$-action respects the filtration on $\mc V$. Denote the subspace of all vectors in  $\mc V$ invariant with respect to the modified $\fkS_n$-action by $\mc V^\fkS$.

Clearly, the $\gl_{1|1}$-action on $\mc V$ commutes with the modified $\fkS_n$-action on $\mc V$ and preserves the grading. Therefore, $\mc V^\fkS$ is a filtered $\gl_{1|1}$-module. Hence we have the weight decomposition for both $\mc V^\fkS$ and $(\mc V^\fkS)^\sing$,
\[
\mc V^\fkS=\bigoplus_{\ell=0}^{n}(\mc V^\fkS)_{(n-\ell,\ell)},\qquad (\mc V^\fkS)^\sing=\bigoplus_{\ell=0}^{n}(\mc V^\fkS)^\sing_{(n-\ell,\ell)}.
\]
Note that $(\mc V^\fkS)_{(n-\ell,\ell)}$ and $(\mc V^\fkS)^\sing_{(n-\ell,\ell)}$ are also filtered $\C[z_1,\dots,z_n]^\fkS$-modules.

\subsection{Properties of $\mc{V}^\fkS$ and $(\mc V^\fkS)^\sing$}
In this section, we describe properties of $\mc{V}^\fkS$ and $(\mc V^\fkS)^\sing$.
\begin{lem}\label{lem free}
The space $(\mc V^\fkS)_{(n-\ell,\ell)}$ is a free $\C[z_1,\dots,z_n]^\fkS$-module of rank $n\choose{\ell}$. In particular, the space $\mc V^\fkS$ is a free $\C[z_1,\dots,z_n]^\fkS$-module of rank $2^n$.
\end{lem}
\begin{proof}
The lemma is proved in Section \ref{sec proof vs}.
\end{proof}

\begin{lem}\label{lem free sing}
The space $(\mc V^\fkS)^\sing_{(n-\ell,\ell)}$ is a free $\C[z_1,\dots,z_n]^\fkS$-module of rank ${n-1}\choose{\ell}$. In particular, the space $(\mc V^\fkS)^\sing$ is a free $\C[z_1,\dots,z_n]^\fkS$-module of rank $2^{n-1}$.
\end{lem}
\begin{proof}
The statement is proved in Section \ref{sec proof vs}.
\end{proof}

For a $\Z_{\gge 0}$-filtered space $M$ with finite-dimensional graded components $\mathscr F_{r}M/\mathscr F_{r-1}M$. Let $\mathrm{gr}(M)$ be the $\Z_{\gge 0}$-grading on $M$ induced from this filtration. Then the graded character of $\mathrm{gr}(M)$, see \eqref{eq grade ch}, is given by
\[
\ch(\mathrm{gr}(M))=\sum_{r=0}^{\infty} (\dim(\mathscr F_{r}M/\mathscr F_{r-1}M))q^r.
\]
Set $(q)_r=\prod_{i=1}^r(1-q^i)$.

\begin{prop}\label{prop ch}
We have
\[
\ch\big(\mathrm{gr}((\mc V^\fkS)_{(n-\ell,\ell)})\big)=\frac{q^{\ell(\ell-1)/2}}{(q)_{\ell}(q)_{n-\ell}},\qquad \ch\big(\mathrm{gr}((\mc V^\fkS)^\sing_{(n-\ell,\ell)})\big)=\frac{q^{\ell(\ell+1)/2}}{(q)_\ell(q)_{n-1-\ell}(1-q^n)}.
\]
\end{prop}
\begin{proof}
The statement is proved in Section \ref{sec proof vs}.
\end{proof}

Consider $\C^{1|1}\otimes V=\C^{1|1}\otimes(\C^{1|1})^{\otimes n}$ and label the factors by $0,1,2,\dots,n$. Define a polynomial $L(x)\in \End(\C^{1|1}\otimes V)\otimes \C[x,z_1,\dots,z_n]$ by
\[
L(x)=(x-z_n+P^{(0,n)})\cdots(x-z_1+P^{(0,1)}).
\]
Consider $L(x)$ as a $2\times 2$ matrix with entries $L_{ij}(x)\in \End(V)\otimes \C[x,z_1,\dots,z_n]$, $i,j=1,2$.

Define the assignment
\beq\label{eq gamma}
\Gamma:T_{ij}(x)\mapsto L_{ij}(x)\prod_{r=1}^n(x-z_r)^{-1}, \quad i,j=1,2.
\eeq
Here we consider $L_{ij}(x)\prod_{r=1}^n(x-z_r)^{-1}$ as a formal power series in $x^{-1}$ whose coefficients are in $\End(V)\otimes \C[z_1,\dots,z_n]$.

\begin{lem}	
The map $\Gamma$ defines a $\Yone$-action on $\mc V$. Moreover, the $\Yone$-action on $\mc V$ preserves the filtration, $\mathscr F_r\Yone\times \mathscr F_s \mc V\to \mathscr F_{r+s}\mc V$ for any $r,s\in\Z_{\gge 0}$.
\end{lem}
\begin{proof}
The first statement follows from the Yang-Baxter equation \eqref{eq yangbaxter}, c.f. \cite[Lemma 3.1]{MTV14}. The second statement is clear.
\end{proof}

\begin{lem}\label{lem com sn}
The $\Yone$-action on $\mc V$ defined by $\Gamma$ commutes with the modified  $\fkS_n$-action \eqref{eq sn action} on $\mc V$ and with multiplication by the elements of $\C[z_1,\dots, z_n]$.
\end{lem}
\begin{proof}
The first statement follows again from the Yang-Baxter equation \eqref{eq yangbaxter}, c.f. \cite[Lemma 3.3]{MTV14}. The second statement is straightforward.
\end{proof}

Therefore, it follows from Lemma \ref{lem com sn} that the space $\mc V^\fkS$ is a filtered $\Yone$-module.

\begin{lem}\label{lem cyclic sym}
The $\Yone$-module $\mc V^\fkS$ is a cyclic module generated by $v_1^{\otimes n}=v_1\otimes\dots\otimes v_1$.
\end{lem}
\begin{proof}
	The lemma is proved in Section \ref{sec proof vs}.
\end{proof}

Given $\bm a=(a_1,\dots,a_n)\in\C^n$, let $I^\fkS_{\bm a}$ be the ideal of $\C[z_1,\dots,z_n]^\fkS$ generated by $\sigma_i(\bm z)-\sigma_i(\bs a)$, $i=1,\dots,n$. Then for any $\bm a$, by Lemmas \ref{lem free} and \ref{lem com sn}, the quotient space $\mc V^\fkS/I^\fkS_{\bm a}\mc V^\fkS$ is a $\Yone$-module of dimension $2^n$ over $\C$. 

\begin{prop}\label{prop local weyl}
Assume that $\bm{a}$ is ordered such that $a_i\ne a_j+1$ for $i>j$. Then the $\Yone$-module $\mc V^\fkS/I^\fkS_{\bm a}\mc V^\fkS$ is isomorphic to $\Yone$-module $V(\bm{a})=\C^{1|1}(a_1)\otimes\dots\otimes\C^{1|1}(a_n)$.
\end{prop}
\begin{proof}
Thanks to Lemmas \ref{lem cyclic weyl}, \ref{lem free}, and \ref{lem cyclic sym}, the proof is similar to that of \cite[Proposition 3.5]{MTV14}.
\end{proof}

\subsection{Proofs of properties}\label{sec proof vs}
We start with Lemma \ref{lem free}.

Define the {\it modified} $\fkS_n$-action on scalar functions in $z_1,\dots,z_n$ by the rule, see e.g. \cite[formula (2.1)]{GRTV}, c.f. also \eqref{eq sn action},
\[
\hat{s}_i:f(z_1,\dots,z_n)\mapsto f(z_1,\dots,z_{i+1},z_i,\dots,z_n)-\frac{f(z_1,\dots,z_n)-f(z_1,\dots,z_{i+1},z_i,\dots,z_n)}{z_i-z_{i+1}}.
\]
Let $\bm f(z_1,\dots,z_n)$ be a $(V)_{(n-\ell,\ell)}$-valued function with coordinates $\{f_I(z_1,\dots,z_n)~|~I\in \mc{I}_\ell\}$,
\[
\bm f(z_1,\dots,z_n)=\sum_{I\in \mc{I}_\ell}f_I(z_1,\dots,z_n)\, v_I.
\]

For $\sigma\in\fkS_n$ and $I=(i_1,\dots,i_n)$, define the natural $\fkS_n$-action on $\mc I_\ell$
\[
\sigma: \mc I_\ell \to \mc I_\ell,\qquad I=(i_1,\dots,i_n)\mapsto \sigma (I)=(i_{\sigma^{-1}(1)},\dots,i_{\sigma^{-1}(n)}).
\]

\begin{lem}\label{lem sn inv case}
The function $\bm f(z_1,\dots,z_n)$ is invariant under modified $\fkS_n$-action \eqref{eq sn action} if and only if for any simple reflection $s_j$ and $I=(i_1,\dots,i_n)$, we have
\[
f_{s_j(I)}=\begin{cases}
-\hat{s}_jf_I,\quad &\text{\emph{if} }i_{j}=i_{j+1}=2,\\
\hat{s}_jf_I,\quad &\text{\emph{otherwise}}.
\end{cases}\qedd
\]
\end{lem}

A two-partition $I\in \mc I_\ell$ corresponds to a permutation $\sigma_{I}\in\fkS_{n}$ as follows:
\beq\label{eq sigma}
(\sigma_{I})^{-1}(j)=\begin{cases}
\#\{r~|~r\lle j,~i_r=2\},&\text{ if } i_j=2,\\
\ell+\#\{r~|~r\lle j,~i_r=1\},&\text{ if } i_j=1.
\end{cases}
\eeq
Set $I^{\max}=(2,\dots,2,1,\dots,1)\in \mc I_\ell$. Then $\sigma_I(I^{\max})=I$.

For any $f\in \C[z_1,\dots,z_n]^{\fkS_{\ell}\times \fkS_{n-\ell}}$, let
\[
\check f=f\cdot\prod_{1\lle r<s\lle \ell}(z_s-z_r-1).
\]
Clearly, $\check f$ is anti-symmetric with respect to the modified action of the subgroup $\fkS_\ell\times 1\subset \fkS_n$ and symmetric with respect to the modified action of the subgroup $1\times\fkS_{n-\ell}\subset \fkS_n$.
\begin{proof}[Proof of Lemma \ref{lem free}]
Define the map
\[
\vartheta_\ell:\C[z_1,\dots,z_n]^{\fkS_{\ell}\times \fkS_{n-\ell}}\to (\mc V^\fkS)_{(n-\ell,\ell)},\qquad
\vartheta_\ell(f)=\sum_{I\in\mc I_\ell} \hat{\sigma}_I(\check f)\, v_{I}.
\]
We show that the map $\vartheta_\ell$ is a well-defined $\C[z_1,\dots,z_n]^\fkS$-module isomorphism. 

We first show that $\vartheta_\ell$ is well-defined. Clearly, $\vartheta_\ell(f)\in (\mc V)_{(n-\ell,\ell)}$. Hence it suffices to show that $\vartheta_\ell(f)\in\mc V^\fkS$. By Lemma \ref{lem sn inv case}, it reduces to show that for any $1\lle j\lle n-1$,
\beq\label{eq needshow}
\hat{\sigma}_{s_j(I)}(\check f)=\begin{cases}
	-\hat{s}_j\hat{\sigma}_{I}(\check f),\quad &\text{if }i_{j}=i_{j+1}=2,\\
	\hat{s}_j\hat{\sigma}_{I}(\check f),\quad &\text{otherwise}.
\end{cases}
\eeq

If $s_j(I)=I$, then $i_j=i_{j+1}$ and $\sigma_{s_j(I)}=\sigma_I$. By \eqref{eq sigma} we have
\[
(\sigma_I)^{-1}\cdot s_j\cdot \sigma_I =((\sigma_I)^{-1}(j),(\sigma_I)^{-1}(j+1))\in\begin{cases}
\fkS_{\ell}\times 1,&\text{ if }i_{j}=i_{j+1}=2,\\
1\times\fkS_{n-\ell},&\text{ if }i_{j}=i_{j+1}=1.
\end{cases}
\]Since $\check f$ is anti-symmetric with respect to the modified $\fkS_\ell\times 1$-action and symmetric with respect to the modified $1\times\fkS_{n-\ell}$-action, the equation \eqref{eq needshow} follows. If $s_j(I)\ne I$, then $i_j\ne i_{j+1}$. Clearly, we have $\sigma_{s_j(I)}=s_j\sigma_I$. Hence the equation \eqref{eq needshow} also follows. Thus the map $\vartheta_\ell$ is well-defined.

Let $\bm f(z_1,\dots,z_n)\in(\mc V^\fkS)_{(n-\ell,\ell)}$,
\[
\bm f(z_1,\dots,z_n)=\sum_{I\in \mc{I}_\ell}f_I(z_1,\dots,z_n)\, v_I.
\]It follows from Lemma \ref{lem sn inv case} that $\bm f$ is uniquely determined by $f_{I^{\max}}$. Moreover, $f_{I^{\max}}$ is anti-symmetric with respect to the modified $\fkS_\ell\times 1$-action and symmetric with respect to the modified $1\times\fkS_{n-\ell}$-action. A direct computation implies that
\[
f_{I^{\max}}\cdot\prod_{1\lle r<s\lle \ell}\frac{1}{z_s-z_r-1}\in \C[z_1,\dots,z_n]^{\fkS_{\ell}\times \fkS_{n-\ell}}.
\]
Therefore $\vartheta_\ell$ is a $\C[z_1,\dots,z_n]^\fkS$-module isomorphism. The lemma now follows from the fact that the algebra $\C[z_1,\dots,z_n]^{\fkS_{\ell}\times \fkS_{n-\ell}}$ is a free $\C[z_1,\dots,z_n]^\fkS$-module of rank $n\choose{\ell}$.
\end{proof}

Consider Lemma \ref{lem cyclic sym}. By the proof of Lemma \ref{lem free}, it suffices to show that for any polynomial $f\in\C[z_1,\dots,z_n]^{\fkS_{\ell}\times \fkS_{n-\ell}}$ one can obtain the vector $\vartheta_\ell(f)$ from the vector $v_1^{\otimes n}$. We proceed by induction on the degree of $f$.
Since $\mathrm{gr}\Yone\cong \mathrm{U}(\gl_{1|1}[t])$, see Section \ref{sec yangian}, 
it is not hard to see it is sufficient to prove the statement for the action of current algebra $\gl_{1|1}[t]$ on $\mc V$.

Define the {\it standard} $\fkS_n$-action on $\mc V$ by the rule:
\[
s_i:\bm f(z_1,\dots,z_n)\mapsto P^{(i,i+1)}\bm f(z_1,\dots,z_{i+1},z_i,\dots,z_n).
\]Denote the subspace of all vectors in $\mc V$ invariant with respect to the standard $\fkS_n$-action by $\mc V^S$. 

The space $\mc V$ is a $\gl_{1|1}[t]$-module where $e_{ij}[r]$ acts by
\begin{align*}
e_{ij}[r](p(z_1,\dots,z_n)&w_1\otimes\dots\otimes w_n)\\=\ &p(z_1,\dots,z_n)\sum_{s=1}^n(-1)^{(|w_1|+\cdots+|w_{s-1}|)(|i|+|j|)}z_s^r\, w_1\otimes\dots\otimes e_{ij}w_s\otimes\dots\otimes w_n,
\end{align*}
for $p(z_1,\dots,z_n)\in\C[z_1,\dots,z_n]$ and $w_s\in \C^{1|1}$.

The $\gl_{1|1}[t]$-action on $\mc V$ commutes with the standard $\fkS_n$-action on $\mc V$, therefore $\mc V^S$ is a $\gl_{1|1}[t]$-module.

The proof of the following lemma is similar to that of Lemma \ref{lem free}.
\begin{lem}\label{lem free graded}
The space $(\mc V^S)_{(n-\ell,\ell)}$ is a free $\C[z_1,\dots,z_n]^\fkS$-module of rank $n\choose{\ell}$. In particular, the space $\mc V^S$ is a free $\C[z_1,\dots,z_n]^\fkS$-module of rank $2^n$.\qed
\end{lem}

\begin{lem}\label{lem cyclic sym graded}
The $\gl_{1|1}[t]$-module $\mc V^S$ is a cyclic module generated by $v_1^{\otimes n}=v_1\otimes\dots\otimes v_1$.
\end{lem}
\begin{proof}
The proof is similar to that of \cite[Lemma 2.11]{MTV09}.
\end{proof}

\begin{lem}\label{lem explicit basis}
The set 
\beq\label{eq explicit basis}
\{e_{21}[r_1]e_{21}[r_2]\cdots e_{21}[r_\ell]v^+~|~0\lle r_1<r_2<\dots<r_{\ell}\lle n-1\}
\eeq
is a free generating set of $(\mc V^S)_{(n-\ell,\ell)}$ over $\C[z_1,\dots,z_n]^\fkS$.\qed
\end{lem}
\begin{proof}
We use the shorthand notation $v^+$ for $v_1^{\otimes n}$. Clearly, we have
\[
e_{21}[r]v^+=\sum_{s=1}^n z_s^r v_1\otimes \dots\otimes v_2\otimes\dots\otimes v_1,
\]where $v_2$ is in the $s$-th factor. Note that
\[
z_s^n+\sum_{i=1}^n(-1)^i\sigma_i(\bm z)z_s^{n-i}=0
\]
and $\{1,z_s,\dots,z_s^{n-1}\}$ is linearly independent over $\C[z_1,\dots,z_n]^\fkS$, therefore
\[
e_{21}[n]v^+=\sum_{i=1}^n(-1)^{i-1}\sigma_i(\bm z)e_{21}[n-i]v^+
\] and $\{e_{21}[0]v^+,e_{21}[1]v^+,\dots,e_{21}[n-1]v^+\}$ is linearly independent over $\C[z_1,\dots,z_n]^\fkS$. Similarly, one shows that each $e_{21}[r]v^+$ is spanned by $e_{21}[0]v^+,e_{21}[1]v^+,\dots,e_{21}[n-1]v^+$ over $\C[z_1,\dots,z_n]^\fkS$. By Lemma \ref{lem free graded}, $(\mc V^S)_{(n-1,1)}$ is free over $\C[z_1,\dots,z_n]^\fkS$ of rank $n$, therefore
\[
e_{21}[0]v^+,e_{21}[1]v^+,\dots,e_{21}[n-1]v^+
\]
are free generators of $(\mc V^S)_{(n-1,1)}$ over $\C[z_1,\dots,z_n]^\fkS$. 

Since $e_{21}[r]e_{21}[s]=-e_{21}[s]e_{21}[r]$, the case of general $\ell$ is proved similarly.
\end{proof}

\begin{lem}\label{lem free sing graded}
The space $(\mc V^S)^\sing_{(n-\ell,\ell)}$ is a free $\C[z_1,\dots,z_n]^\fkS$-module of rank ${n-1}\choose{\ell}$ with a free generating set given by
\beq\label{eq basis}
\{e_{12}[0]e_{21}[0]e_{21}[r_1]\cdots e_{21}[r_{\ell}]v^+,\quad 1\lle r_1<r_2<\dots<r_{\ell}\lle n-1\}.
\eeq
In particular, the space $(\mc V^S)^\sing$ is a free $\C[z_1,\dots,z_n]^\fkS$-module of rank $2^{n-1}$.
\end{lem}
\begin{proof}
Let $w$ be a non-zero singular vector in  $(\mc V^S)^\sing_{(n-\ell,\ell)}$, then 
\[
e_{12}[0]e_{21}[0]w=(e_{11}[0]+e_{22}[0])w-e_{21}[0]e_{12}[0]w=(e_{11}[0]+e_{22}[0])w=nw.
\]
By Lemma \ref{lem explicit basis}, we have that $w$ is a linear combination of vectors in \eqref{eq basis} over $\C[z_1,\dots,z_n]^\fkS$. 

Now it suffices to show that the set given in \eqref{eq basis} is linearly independent over $\C[z_1,\dots,z_n]^\fkS$. Note that
\begin{align*}
e_{12}[0]e_{21}[0]e_{21}[r_1]\cdots e_{21}[r_{\ell}]v^+=\ &ne_{21}[r_1]\cdots e_{21}[r_{\ell}]v^+ \\+&\sum_{i=1}^\ell (-1)^i p_{r_i}(\bm z)e_{21}[0]e_{21}[r_1]\cdots \widehat{e_{21}[r_i]}\cdots e_{21}[r_{\ell}]v^+,
\end{align*}
where $p_j(\bm z)=\sum_{s=1}^n z_s^j$ and $\widehat{e_{21}[r_i]}$ means the factor $e_{21}[r_i]$ is skipped. Therefore the statement follows from Lemma \ref{lem explicit basis}.
\end{proof}

We are ready to prove Proposition \ref{prop ch}. Note that 
\[
\mathrm{gr}((\mc V^\fkS)_{(n-\ell,\ell)})=(\mc V^S)_{(n-\ell,\ell)}, \qquad \mathrm{gr}((\mc V^\fkS)^\sing_{(n-\ell,\ell)})=(\mc V^S)^\sing_{(n-\ell,\ell)}.
\]
\begin{proof}[Proof of Proposition \ref{prop ch}]
By Lemma \ref{lem free graded}, the graded character $\ch\big((\mc V^S)_{(n-\ell,\ell)}\big)$ is equal to the graded character of the space spanned by the set \eqref{eq explicit basis} over $\C$ multiplying with the graded character of $\C[z_1,\dots,z_n]^\fkS$. We have $\ch\big(\C[z_1,\dots,z_n]^\fkS\big)=1/(q)_n$,
therefore 
\[
\ch\big(\mathrm{gr}((\mc V^\fkS)_{(n-\ell,\ell)})\big)=\ch\big((\mc V^S)_{(n-\ell,\ell)}\big)=q^{\ell(\ell-1)/2}\cdot\frac{(q)_n}{(q)_\ell(q)_{n-\ell}}\cdot\frac{1}{(q)_n}=\frac{q^{\ell(\ell-1)/2}}{(q)_{\ell}(q)_{n-\ell}}.
\]
Similarly, by Lemma \ref{lem free sing graded},
\[
\ch\big(\mathrm{gr}((\mc V^\fkS)^\sing_{(n-\ell,\ell)})\big)=\ch\big((\mc V^S)^\sing_{(n-\ell,\ell)}\big)=q^{\ell(\ell+1)/2}\cdot\frac{(q)_{n-1}}{(q)_\ell(q)_{n-1-\ell}}\cdot\frac{1}{(q)_n}=\frac{q^{\ell(\ell+1)/2}}{(q)_\ell(q)_{n-1-\ell}(1-q^n)}.\qedhere
\]
\end{proof}

\section{Main theorems}\label{sec main thms}
\subsection{The algebra $\mc{O}_{l}$}\label{sec o_l}
Let $\Omega_{l}$ be the $n$-dimensional affine space with coordinates $f_1,\dots,f_l$, $g_1$, $\dots$, $g_{n-l-1}$ and $\vSi_n$. Introduce two polynomials 
\beq\label{eq poly f g}
f(x)=x^l+\sum_{i=1}^{l}f_ix^{l-i},\quad g(x)=x^{n-l-1}+\sum_{i=1}^{n-l-1}g_ix^{n-l-i-1}.
\eeq
Denote by $\mc O_l$ the algebra of regular functions on $\Omega_{l}$, namely $\mc O_l=\C[f_1,\dots,f_l,g_1,\dots,g_{n-l-1},\vSi_n]$. Define the degree function by
\[
\deg f_i=i,\qquad \deg g_j=j,\qquad \deg\vSi_n=n,
\]for all $i=1,\dots,l$ and $j=1,\dots,n-l-1$. The algebra $\mc O_l$ is graded with the graded character given by
\beq\label{eq ch O}
\ch(\mc O_l)=\frac{1}{(q)_l(q)_{n-l-1}(1-q^n)}.
\eeq
Let $\mathscr F_0\mc O_l\subset \mathscr F_1\mc O_l\subset\dots\subset \mc O_l$ be the increasing filtration corresponding to this grading, where $\mathscr F_s\mc O_l$ consists of elements of degree at most $s$.

Let $\varSigma_1,\dots,\varSigma_{n-1}$ be the elements of $\mc O_l$ such that
\beq\label{eq si q=1}
nf(x)g(x)=\Big((x+1)^n+\sum_{i=1}^{n-1}(-1)^i\varSigma_i(x+1)^{n-i}\Big)-\Big(x^n+\sum_{i=1}^{n-1}(-1)^i\varSigma_ix^{n-i}\Big).
\eeq
The homomorphism
\beq\label{eq inj sym}
\pi_l:\C[z_1,\dots,z_n]^\fkS\to \mc O_l,\qquad \sigma_i(\bm z)\mapsto \vSi_i,\qquad i=1,\dots,n.
\eeq
is injective and induces a $\C[z_1,\dots,z_n]^\fkS$-module structure on $\mc O_l$.


Express $nf(x-1)g(x)$ as follows,
\beq\label{eq G q=1}
nf(x-1)g(x)=nx^{n-1}+\sum_{i=1}^{n-1}G_{i}x^{n-1-i},
\eeq
where $G_{i}\in\mc O_l$. 

\begin{lem}
The elements $G_{i}$ and $\varSigma_j$, $i=1,\dots,n-1$, $j=1,\dots, n$, generate the algebra $\mc O_l$.\qed
\end{lem}

\begin{lem}\label{lem deg O}
We have $G_{i}\in \mathscr F_{i}\mc O_l\setminus \mathscr F_{i-1}\mc O_l$ and $\varSigma_j\in \mathscr F_{j}\mc O_l\setminus \mathscr F_{j-1}\mc O_l$ for $i=1,\dots,n-1$, $j=1,\dots, n$.\qed
\end{lem}

\subsection{The algebra $\mc{O}_{l}^{Q}$}
We rework Section \ref{sec o_l} for the case of $q_1\neq q_2$.

Let $q_1\ne q_2$. Let $\Omega_{l}^Q$ be the $n$-dimensional affine space with coordinates $f_1,\dots,f_l$ and $g_1$, $\dots$, $g_{n-l}$. Introduce two polynomials 
\[
f(x)=x^l+\sum_{i=1}^{l}f_ix^{l-i},\quad g(x)=x^{n-l}+\sum_{i=1}^{n-l}g_ix^{n-l-i}.
\]
Denote by $\mc O^Q_l$ the algebra of regular functions on $\Omega_{l}^Q$, namely $\mc O^Q_l=\C[f_1,\dots,f_l,g_1,\dots,g_{n-l}]$. Let $\deg f_i=i$ and $\deg g_j=j$ for all $i=1,\dots,l$ and $j=1,\dots,n-l$. The algebra $\mc O_l^Q$ is graded with the graded character
\beq\label{eq char Q}
\ch(\mc O_l^Q)=\frac{1}{(q)_l(q)_{n-l}}.
\eeq
Let $\mathscr F_0\mc O_l^Q\subset \mathscr F_1\mc O_l^Q\subset\dots\subset \mc O_l^Q$ be the increasing filtration corresponding to this grading, where $\mathscr F_s\mc O_l^Q$ consists of elements of degree no greater than $s$.

Let $\varSigma_1,\dots,\varSigma_n$ be the elements of $\mc O_l^Q$ such that
\[
(q_1-q_2)f(x)g(x)=q_1\Big((x+1)^n+\sum_{i=1}^n(-1)^i\varSigma_i(x+1)^{n-i}\Big)-q_2\Big(x^n+\sum_{i=1}^n(-1)^i\varSigma_ix^{n-i}\Big).
\]
The homomorphism
\[
\pi^Q_l:\C[z_1,\dots,z_n]^\fkS\to \mc O^Q_l,\qquad \sigma_i(\bm z)\mapsto \vSi_s,\qquad i=1,\dots,n.
\]
is injective and induces a $\C[z_1,\dots,z_n]^\fkS$-module structure on $\mc O_l^Q$.


Express $(q_1-q_2)f(x-1)g(x)$ as follows,
\[
(q_1-q_2)f(x-1)g(x)=(q_1-q_2)x^{n}+\sum_{i=1}^{n}G_{i}^Qx^{n-i},
\]
where $G_{i}^Q\in\mc O_l^Q$.

\begin{lem}
The elements $G^Q_{i}$ and $\varSigma_i$, $i=1,\dots, n$, generate the algebra $\mc O^Q_l$.\qed
\end{lem}
\begin{lem}\label{lem deg O Q}
	We have $G^Q_{i}\in \mathscr F_{i}\mc O^Q_l\setminus \mathscr F_{i-1}\mc O^Q_l$ and $\varSigma_i\in \mathscr F_{i}\mc O^Q_l\setminus \mathscr F_{i-1}\mc O^Q_l$ for $i=1,\dots, n$.\qed
\end{lem}
\subsection{Berezinian and Bethe algebra}\label{sec ber}
We follow the convention of \cite{MR14}. Let $\mc A$ be a superalgebra. Consider the operators of the form
\[
\mc K=
\sum_{i,j=1}^2(-1)^{|i||j|+|j|}E_{ij}\otimes K_{ij}\in \End(\C^{1|1})\otimes \mc A,
\]
where $K_{ij}$ are elements of $\mc A$ of parity $|i|+|j|$. We say that $\mc K$ is a \emph{Manin matrix} if
\[
[K_{ij},K_{rs}]=(-1)^{|i||j|+|i||r|+|j||r|}[K_{rj},K_{is}]
\]
for all $i,j,r,s=1,2$. In other words, $\mc K$ is Manin if and only if
$$
[K_{11},K_{21}]=[K_{22},K_{21}]=0,\qquad   [K_{11},K_{22}]=[K_{12},K_{21}].
$$

Assume that $K_{11},K_{22}$, and $K_{22}-K_{21}K_{11}^{-1}K_{12}$ are all invertible. 
Define the \emph{Berezinian} of $\mc K$ by
\[
\mathrm{Ber}(\mc K)=K_{11}(K_{22}-K_{21}K_{11}^{-1}K_{12})^{-1}.
\]
Then it is a straightforward check, see also \cite[Corollary 2.16]{MR14}, that 
we have three other ways to write the Berezinian:
\beq\label{ber}
\mathrm{Ber}(\mc K)=(K_{22}+K_{12}K_{11}^{-1}K_{21})^{-1}K_{11}=
K_{22}^{-1}(K_{11}-K_{12}K_{22}^{-1}K_{21})=(K_{11}+K_{21}K_{22}^{-1}K_{12})K_{22}^{-1}.
\eeq

Let $\tau$ be the difference operator, $(\tau f)(x)=f(x-1)$ for any function $f$ in $x$. Let $\mc A$ be the superalgebra $\Yone[\tau]$, where $\tau$ has parity $\bar 0$. Consider the operator $Z^Q(x,\tau)$,
\[
Z^Q(x,\tau)=T^t(x)Q^t\tau=T^t(x)Q\tau\in \End(\C^{1|1})\otimes\Yone[\tau].
\]
It follows from \eqref{eq RTT} or \eqref{com relations} that $Z^Q(x,\tau)$ is a Manin matrix, see e.g. \cite[Remark 2.12]{MR14}. Note that our generating series $T_{ij}(x)$ corresponds to $z_{ji}(u)$ in \cite{MR14}.

Denote $\mathrm{Ber}^Q(x)=\mathrm{Ber}(Z^Q(x,\tau))$. Then $\mathrm{Ber}^Q(x)$ is a scalar operator (does not contain $\tau$), moreover, $\mathrm{Ber}^Q(x)q_2/q_1$ does not depend on $Q$.

Expand $\mathrm{Ber}^Q(x)$ as a power series in $x^{-1}$ with coefficients in $\Yone$. The following proposition was conjectured by Nazarov \cite[Conjecture 2]{Naz} and proved by Gow \cite[Theorem 4]{Gow} for the general case of $\YglMN$.
\begin{prop}\label{prop central}
The coefficients of $\mathrm{Ber}^Q(x)$ generate the center of $\Yone$.\qed
\end{prop}

We call the subalgebra of $\Yone$ generated by the coefficients of $\mathrm{Ber}^Q(x)$ and $\mc T_Q(x)$ the \emph{Bethe algebra associated to $Q$}, c.f. Remark \ref{two Bethe algebras}. We denote the Bethe algebra associated to $Q$ by $\mc B^Q$. We simply write $\mc B$ for $\mc B^I$, where $I$ is the identity matrix.

Let $\bla=(\la^{(1)},\dots,\la^{(k)})$ be a sequence of polynomial $\gl_{1|1}$-weights. Let $\bm a=(a_1,\dots,a_n)$ and $\bm b=(b_1,\dots,b_k)$ be two sequences of complex numbers. Recall that we have three kinds of modules, $\mc V^\fkS$,  $V(\bm{a})=\C^{1|1}(a_1)\otimes\dots\otimes\C^{1|1}(a_n)$, and $L(\bla,\bm b)=L_{\la^{(1)}}(b_1)\otimes \dots\otimes L_{\la^{(k)}}(b_k)$. 

Our main problem is to understand the spectrum of the Bethe algebra acting on $L(\bla,\bm b)$, when $L(\bla,\bm b)$ is a cyclic $\Yone$-module.

If $q_1\ne q_2$, let $\mc B_l^Q(\bm z)$, $\mc B_l^Q(\bm a)$, and $\mc B_l^Q(\bla,\bm b)$ denote, respectively, the images of the Bethe algebra $\mc B^Q$ in $\End((\mc V^\fkS)_{(n-l,l)})$, $\End((V(\bm{a}))_{(n-l,l)})$, and $\End((L(\bla,\bm b))_{(n-l,l)})$. For any element $X^Q\in\mc B^Q$, we denote by $X^Q(\bm z),X^Q(\bm a),X^Q(\bla,\bm b)$ the respective linear operators. 

If $q_1=q_2=1$, let $\mc B_l(\bm z)$, $\mc B_l(\bm a)$, and $\mc B_l(\bla,\bm b)$ denote, respectively, the images of the Bethe algebra $\mc B$ in $\End((\mc V^\fkS)^\sing_{(n-l,l)})$, $\End((V(\bm{a}))^\sing_{(n-l,l)})$, and $\End((L(\bla,\bm b))^\sing_{(n-l,l)})$. For any element $X\in \mc B$, we denote by $X(\bm z),X(\bm a),X(\bla,\bm b)$ the respective linear operators.

By abuse of language, we call the image of Bethe algebra again Bethe algebra.

Since by Lemma \ref{lem cyclic sym} the $\Yone$-module $\mc V^\fkS$ is generated by $v_1^{\otimes n}=v_1\otimes\dots\otimes v_1$, the series $\mathrm{Ber}^Q(x)$ acts on $\mc V^\fkS$ by multiplication by the series
\[
\frac{q_1}{q_2}\cdot\frac{(x-z_1+1)\cdots(x-z_n+1)}{(x-z_1)\cdots(x-z_n)}.
\]
Therefore there exist uniquely central elements $C_1,\dots,C_n$ of $\Yone$ of minimal degrees such that each $C_i$ acts on $\mc V^\fkS$ by multiplication by $\sigma_i(\bm z)$. 

Define $B_{i}^Q\in \mc B^Q$ by
\beq\label{eq diff poly}
\Big(x^n+\sum_{i=1}^n(-1)^iC_ix^{n-i}\Big)\mc T_Q(x)=x^{n}\Big((q_1-q_2)+\sum_{i=1}^{\infty} B_i^Qx^{-i}\Big).
\eeq
We write simply $B_i$ for $B_i^I$, where $I$ is the identity matrix.

Recall the $\Yone$-action on $\mc V^\fkS$ from \eqref{eq gamma}, we have $B_i^Q(\bm z)=0$ for $i> n$. When $Q=I$, we have $B_1(\bm z)=n$.

\begin{lem}\label{lem gen B}
The elements $B_{i}^Q(\bm z)$ and $C_i(\bm z)$, $i=1,\dots, n$, generate the algebra $\mc B_l^Q(\bm z)$.\qed
\end{lem}

\begin{lem}\label{lem deg B}
We have $C_i\in \mathscr F_i \mc B^Q\setminus \mathscr F_{i-1} \mc B^Q$, $B_{i}\in \mathscr F_{i-1}\mc B\setminus \mathscr F_{i-2} \mc B$, and $B^Q_{i}\in \mathscr F_{i}\mc B^Q\setminus \mathscr F_{i-1} \mc B^Q$ for $i=1,\dots, n$.\qed
\end{lem}

\subsection{Main theorems for the case $q_1=q_2$}\label{sec main thm}
Recall from Proposition \ref{prop ch} that there exists a unique vector (up to proportionality) of degree $l(l+1)/2$ in $(\mc V^\fkS)^\sing_{(n-l,l)}$ explicitly given by $\mathfrak u_l:=T_{21}^{(1)}T_{12}^{(1)}T_{12}^{(2)}\cdots T_{12}^{(l+1)}v^+$, see Lemma \ref{lem free sing graded}.

Any commutative algebra $\mc A$ is a module over itself induced by left multiplication. We call it the \emph{regular representation of} $\mc A$. The dual space 
$\mc A^*$ is naturally an $\mc A$-module which is called the \emph{coregular representation}. A bilinear form $(\cdot|\cdot):\mc A\otimes \mc A\to \C$ is called \emph{invariant} if $(ab|c)=(a|bc)$ for all $a,b,c\in\mc A$. A finite-dimensional commutative algebra $\mc A$ admitting an invariant non-degenerate symmetric bilinear form $(\cdot|\cdot):\mc A\otimes \mc A\to \C$ is called a \emph{Frobenius algebra}. The regular and coregular representations of a  Frobenius algebra are isomorphic.

Let $M$ be an $\mathcal A$-module and $\mathcal E:\mathcal A\to \C$ a character, then the {\it $\mc A$-eigenspace associated to $\mc E$} in $M$ is defined by $\bigcap_{a\in \mathcal A}\ker(a|_M-\mathcal E(a))$. The {\it generalized $\mc A$-eigenspace associated to $\mc E$} in $M$ is defined by $\bigcap_{a\in \mathcal A}\big(\bigcup_{m=1}^\infty\ker(a|_M-\mathcal E(a))^m\big)$.

\begin{thm}\label{thm VS}
The action of the  Bethe algebra $\mc B_l(\bm z)$ on $(\mc V^\fkS)_{(n-l,l)}^\sing$ has the following properties.
\begin{enumerate}
\item The map $\eta_l:G_{i}\mapsto B_{i+1}(\bm z)$, $\varSigma_j\mapsto C_j(\bm z)$, $i=1,\dots,n-1$, $j=1,\dots,n$, extends uniquely to an isomorphism $\eta_l:\mc O_l\to \mc B_l(\bm z)$ of filtered algebras. Moreover, the isomorphism $\eta_l$ is an isomorphism of $\C[z_1,\dots,z_n]^\fkS$-modules.
	
\item The map $\rho_l:\mc O_l\mapsto (\mc V^\fkS)_{(n-l,l)}^\sing$, $F\mapsto \eta_l(F)\mathfrak u_l$, is an isomorphism of filtered vector spaces identifying the $\mc B_l(\bm z)$-module $(\mc V^\fkS)_{(n-l,l)}^\sing$ with the regular representation of $\mc O_l$.
\end{enumerate}
\end{thm}

Theorem \ref{thm VS} is proved in Section \ref{sec proof}.

\medskip

Let $\bla=(\la^{(1)},\dots,\la^{(k)})$ be a sequence of polynomial $\gl_{1|1}$-weights such that $|\bla|=n$. Let $\bm b=(b_1,\dots,b_k)$ be a sequence of complex numbers. Suppose the $\Yone$-module $L(\bla,\bm b)$ is cyclic.
\begin{thm}\label{thm tensor irr}
The action of the Bethe algebra $\mc B_l(\bla,\bm b)$ on $(L(\bla,\bm b))^\sing_{(n-l,l)}$ has the following properties.
\begin{enumerate}
\item The Bethe algebra $\mc B_l(\bla,\bm b)$ is isomorphic to\[
\C[w_1,\dots,w_{k-1}]^{\fkS_{l}\times \fkS_{k-l-1}}/\langle \sigma_i(\bm w)-\varepsilon_i\rangle_{i=1,\dots,k-1},\]
where $\varepsilon_i$ are given by $\varphi_{\bla,\bm b}(x)-\psi_{\bla,\bm b}(x)=n(x^{k-1}+\sum_{i=1}^{k-1}(-1)^i\varepsilon_ix^{k-1-i})$ and $\sigma_i(\bm w)$ are elementary symmetric functions in $w_1,\dots,w_{k-1}$. Under this isomorphism, $\prod_{s=1}^k(x-b_s)\mc T(x)$ corresponds to $n\prod_{i=1}^{l}(x-w_i-1)\prod_{j=l+1}^{k-1}(x-w_j)$.
\item The Bethe algebra $\mc B_l(\bla,\bm b)$ is a Frobenius algebra. Moreover, the $\mc B_l(\bla,\bm b)$-module $(L(\bla,\bm b))^\sing_{(n-l,l)}$ is isomorphic to the regular representation of $\mc B_l(\bla,\bm b)$.
\item The Bethe algebra $\mc B_l(\bla,\bm b)$ is a maximal commutative subalgebra in $\End((L(\bla,\bm b))^\sing_{(n-l,l)})$ of dimension $\binom{k-1}{l}$.
\item Every $\mc B$-eigenspace in $(L(\bla,\bm b))^\sing_{(n-l,l)}$ has dimension one.
\item The $\mc B$-eigenspaces in $(L(\bla,\bm b))^\sing_{(n-l,l)}$ bijectively correspond to the monic degree $l$ divisors $y(x)$ of the polynomial $\varphi_{\bla,\bm b}(x)-\psi_{\bla,\bm b}(x)$. Moreover, the eigenvalue of $\cT(x)$ corresponding to the monic divisor $y$ is described by $\mc E^I_{y,\bla,\bm b}(x)$, see \eqref{eq gl11 eigenvalue in y}.
\item Every generalized $\mc B$-eigenspace in $(L(\bla,\bm b))^{\sing}_{(n-l,l)}$ is a cyclic $\mc B$-module.
\item The dimension of the generalized $\mc B$-eigenspace associated to $\mc E^I_{y,\bla,\bm b}(x)$ is 
\[
\prod_{a\in \C}\binom{\mathrm{Mult}_a(\varphi_{\bla,\bm b}-\psi_{\bla,\bm b})}{\mathrm{Mult}_a(y)},
\]where $\mathrm{Mult}_a(p)$ is the multiplicity of $a$ as a root of the polynomial $p$.
\end{enumerate}
\end{thm}

Theorem \ref{thm tensor irr} is proved in Section \ref{sec proof}.

\subsection{Main theorems for the case $q_1\ne q_2$}\label{sec main thm q}
Recall from Proposition \ref{prop ch} that there exists a unique vector (up to proportionality) of degree $l(l-1)/2$ in $(\mc V^\fkS)_{(n-l,l)}$ explicitly given by $\mathfrak u^Q_l:=T_{12}^{(1)}T_{12}^{(2)}\cdots T_{12}^{(l)}v^+$, see Lemma \ref{lem explicit basis}.

\begin{thm}\label{thm VS q}
The action of the Bethe algebra $\mc B^Q_l(\bm z)$ on $(\mc V^\fkS)_{(n-l,l)}$ has the following properties.
\begin{enumerate}
	\item The map $\eta^Q_l:G^Q_{i}\mapsto B^Q_{i}(\bm z)$, $\varSigma_i\mapsto C_i(\bm z)$, $i=1,\dots, n$, extends uniquely to an isomorphism $\eta^Q_l:\mc O^Q_l\to \mc B_l^Q(\bm z)$ of filtered algebras. Moreover, the isomorphism $\eta^Q_l$ is an isomorphism of $\C[z_1,\dots,z_n]^\fkS$-modules.
	\item The map $\rho^Q_l:\mc O^Q_l\mapsto (\mc V^\fkS)_{(n-l,l)}$, $F\mapsto \eta_l(F)\mathfrak u^Q_l$, is an isomorphism of filtered vector spaces identifying the $\mc B^Q_l(\bm z)$-module $(\mc V^\fkS)_{(n-l,l)}$ with the regular representation of $\mc O^Q_l$.
\end{enumerate}
\end{thm}
Theorem \ref{thm VS q} is proved in Section \ref{sec proof}.

\medskip

Let $\bla=(\la^{(1)},\dots,\la^{(k)})$ be a sequence of polynomial $\gl_{1|1}$-weights such that $|\bla|=n$. Let $\bm b=(b_1,\dots,b_k)$ be a sequence of complex numbers. Suppose the $\Yone$-module $L(\bla,\bm b)$ is cyclic.

\begin{thm}\label{thm tensor irr q}
The action of the Bethe algebra $\mc B^Q_l(\bla,\bm b)$ on $(L(\bla,\bm b))_{(n-l,l)}$ has the following properties.
\begin{enumerate}
	\item The Bethe algebra $\mc B^Q_l(\bla,\bm b)$ is isomorphic to\[
\C[w_1,\dots,w_{k}]^{\fkS_{l}\times \fkS_{k-l}}/\langle \sigma_i(\bm w)-\varepsilon_i\rangle_{i=1,\dots,k},\]
where $\varepsilon_i$ are given by $q_1\varphi_{\bla,\bm b}(x)-q_2\psi_{\bla,\bm b}(x)=(q_1-q_2)(x^{k}+\sum_{i=1}^{k}(-1)^i\varepsilon_ix^{k-i})$ and $\sigma_i(\bm w)$ are elementary symmetric functions in $w_1,\dots,w_{k}$. Under this isomorphism, $\prod_{s=1}^k(x-b_s)\mc T_Q(x)$ corresponds to $ (q_1-q_2)\prod_{i=1}^{l}(x-w_i-1)\prod_{j=l+1}^{k}(x-w_j)$.
	\item The Bethe algebra $\mc B^Q_l(\bla,\bm b)$ is a Frobenius algebra. Moreover, the $\mc B^Q_l(\bla,\bm b)$-module $(L(\bla,\bm b))_{(n-l,l)}$ is isomorphic to the regular representation of $\mc B^Q_l(\bla,\bm b)$.
	\item The Bethe algebra $\mc B^Q_l(\bla,\bm b)$ is a maximal commutative subalgebra in $\End((L(\bla,\bm b))_{(n-l,l)})$ of dimension $\binom{k}{l}$.
    \item Every $\mc B^Q$-eigenspace in $(L(\bla,\bm b))_{(n-l,l)}$ has dimension one.
	\item The  $\mc B^Q$-eigenspaces in $(L(\bla,\bm b))_{(n-l,l)}$ bijectively correspond to the monic degree $l$ divisors $y(x)$ of the polynomial $q_1\varphi_{\bla,\bm b}(x)-q_2\psi_{\bla,\bm b}(x)$. Moreover, the eigenvalue of $\mc T_Q(x)$ corresponding to the monic divisor $y(x)$ is described by $\mc E^Q_{y,\bla,\bm b}(x)$, see \eqref{eq gl11 eigenvalue in y}.
	\item Every generalized $\mc B^Q$-eigenspace in $(L(\bla,\bm b))_{(n-l,l)}$ is a cyclic $\mc B^Q$-module.
\item The dimension of the generalized $\mc B^Q$-eigenspace associated to $\mc E^Q_{y,\bla,\bm b}(x)$ is 
\[
\prod_{a\in \C}\binom{\mathrm{Mult}_a(q_1\varphi_{\bla,\bm b}-q_2\psi_{\bla,\bm b})}{\mathrm{Mult}_a(y)},
\]where $\mathrm{Mult}_a(p)$ is the multiplicity of $a$ as a root of the polynomial $p$.
\end{enumerate}
\end{thm}
Theorem \ref{thm tensor irr q} is proved in Section \ref{sec proof}.

\section{Proof of main theorems}\label{sec proof}
In this section, we prove the main theorems. 
\subsection{The first isomorphism}
\begin{proof}[Proof of Theorem \ref{thm VS}]
We first show the homomorphism defined by $\eta_l$ is well-defined. 

Consider the tensor product $V(\bs a)=\C^{1|1}(a_1)\otimes\dots\otimes \C^{1|1}(a_n)$, where $a_i\in \C$, and the corresponding Bethe ansatz equation associated to weight $(n-l,l)$.
Let $\bs t$ be a solution with distinct coordinates and $\widehat{\bB}_l(\bm t)$ the corresponding on-shell Bethe vector. Denote $\mc E_{i,\bs t}$ the eigenvalues of $B_i$ acting on $\widehat{\bB}_l(\bm t)$, see Theorem \ref{thm gl11 eigenvalue in y}. 

Define a character $\pi:\mc O_l\to \C$ by sending  
$$ f(x)\mapsto y_{\bs t}(x),\qquad g(x)\mapsto\frac1{ny_{\bs t}(x)}\ \Big(\prod_{i=1}^n(x-a_i+1)-\prod_{i=1}^n(x-a_i)\Big),\qquad\varSigma_n\mapsto \prod_{i=1}^n a_i.$$

Then
\beq\label{equal on bethe vector} 
\pi(\varSigma_i)= \sigma_i(\bs a), \qquad \pi(G_i)=\mc E_{i,\bs t},
\eeq
by \eqref{eq si q=1} and by \eqref{eq gl11 eigenvalue in y}, \eqref{eq G q=1},  respectively.

\medskip

Let now $P(G_{i},\varSigma_j)$ be a polynomial in $G_{i},\varSigma_j$ such that $P(G_{i},\varSigma_j)$ is equal to zero in $\mc O_{l}$. It suffices to show $P(B_i(\bm z),C_j(\bm z))$ is equal to zero in $\mc B_l(\bm z)$. 

Note that $P(B_i(\bm z),C_j(\bm z))$ is a polynomial in $z_1,\dots,z_n$ with values in $\End((V)_{(n-l,l)}^\sing)$. For any sequence $\bm a$ of complex numbers, we can evaluate $P(B_i(\bm z),C_j(\bm z))$ at $\bs z=\bs a$ to an operator on $V(\bs a)_{(n-l,l)}^\sing$. By Theorem \ref{thm complete}, the transfer matrix $\mc T(x)$ is diagonalizable and the Bethe ansatz is complete for $(V(\bm a))_{(n-l,l)}^\sing$ when $\bm a\in\C^n$ is generic. Hence by \eqref{equal on bethe vector} the value of $P(B_i(\bm z),C_j(\bm z))$ at $\bs z=\bs a$ is also equal to zero for generic $\bs a$. Therefore $P(B_i(\bm z),C_j(\bm z))$ is identically zero and the map $\eta_l$ is well-defined.

\medskip

Let us now show that the map $\eta_l$ is injective. Let $P(G_{i},\varSigma_j)$ be a polynomial in $G_{i},\varSigma_j$ such that $P(G_{i},\varSigma_j)$ is non-zero in $\mc O_{l}$. Then the value at a generic point of $\Omega_l$ (e.g. the non-vanishing points of $P(G_{i},\varSigma_j)$ such that $f$ and $g$ are relatively prime and have only simple zeros) is not equal to zero. Moreover, at those points the transfer matrix $\mc T(x)$ is diagonalizable and the Bethe ansatz is complete again by Theorem \ref{thm complete}. Therefore, again by \eqref{equal on bethe vector}, the polynomial $P(B_i(\bm z),C_j(\bm z))$ is a non-zero element in $\mc B_l(\bm z)$. Thus the map $\eta_l$ is injective.

The surjectivity of $\eta_l$ follows from Lemma \ref{lem gen B}. Hence $\eta_l$ is an isomorphism of algebras. 

The fact that $\eta_l$ is an isomorphism of filtered algebra respecting the filtration follows from Lemmas \ref{lem deg O} and \ref{lem deg B}. This completes the proof of part (i).

\medskip

The kernel of $\rho_l$ is an ideal of $\mc O_l$. Note that the algebra $\mc O_l$ contains the algebra $\C[z_1,\dots,z_n]^\fkS$ if we identify $\sigma_i(\bm z)$ with $\vSi_i$, see \eqref{eq inj sym}. The kernel of $\rho_l$ intersects $\C[z_1,\dots,z_n]^\fkS$ trivially. Therefore the kernel of $\rho_l$ is trivial as well. Hence $\rho_l$ is an injective map. Comparing \eqref{eq ch O} and Proposition \ref{prop ch}, we have $\ch\big(\mathrm{gr}((\mc V^\fkS)^\sing_{(n-\ell,\ell)})\big)=q^{l(l+1)/2}\ch(\mc O_l)$. Thus $\rho_l$ is an isomorphism of filtered vector spaces which shifts the degree by $l(l+1)/2$, completing the proof of part (ii).
\end{proof}

\begin{proof}[Proof of Theorem \ref{thm VS q}]The proof of Theorem \ref{thm VS q} is similar to that of Theorem \ref{thm VS} with the help of \eqref{eq char Q}.
\end{proof}


\subsection{The second isomorphism}\label{sec sec-iso}

Let $\bm a=(a_1,\dots,a_n)$ be a sequence of complex numbers such that $a_i\ne a_j+1$ for $i>j$. Let $I_{l,\bm a}^{\mc O}$ be the ideal of $\mc O_l$ generated by the elements $\vSi_i-\sigma_i(\bs a)$, $i=1,\dots,n$, where $\vSi_1,\dots,\vSi_{n-1}$ are defined in
\eqref{eq si q=1}. Let $\mc O_{l,\bm a}$ be the quotient algebra
\[
\mc O_{l,\bm a}=\mc O_{l}/I_{l,\bm a}^{\mc O}.
\]

Let $I_{l,\bm a}^{\mc B}$ be the ideal of $\mc B_l(\bm z)$ generated by $C_i(\bm z)-\sigma_i(\bs a)$, $i=1,\dots,n$. Consider the subspace 
\[
I_{l,\bm a}^{\mc M}=I_{l,\bm a}^{\mc B}(\mc V^\fkS)_{(n-l,l)}^\sing=(I^\fkS_{\bm a}\mc V^\fkS)_{(n-l,l)}^\sing,
\]
where $I^\fkS_{\bm a}$ as before is the ideal of $\C[z_1,\dots,z_n]^\fkS$ generated by $\sigma_i(\bm z)-\sigma_i(\bs a)$. 

\begin{lem}\label{lem eta rho weyl}
We have
\[
\eta_l(I_{l,\bm a}^{\mc O})=I_{l,\bm a}^{\mc B},\quad \rho_l(I_{l,\bm a}^{\mc O})=I_{l,\bm a}^{\mc M},\quad \mc B_{l}(\bm a)=\mc B_{l}(\bm z)/I_{l,\bm a}^{\mc B},\quad (V(\bm a))^\sing_{(n-l,l)}=(\mc V^\fkS)^\sing_{(n-l,l)}/ I_{l,\bm a}^{\mc M}.
\]
\end{lem}
\begin{proof}
The lemma follows from Theorem \ref{thm VS} and Proposition \ref{prop local weyl}.
\end{proof}

We prove part (ii) of Theorem \ref{thm tensor irr} for the special case $V(\bm a)$.

By Lemma \ref{lem eta rho weyl}, the maps $\eta_l$ and $\rho_l$ induce the maps
\[
\eta_{l,\bm a}:\mc O_{l,\bm a}\to \mc B_{l}(\bm a),\qquad \rho_{l,\bm a}:\mc O_{l,\bm a}\to (V(\bm a))_{(n-l,l)}^\sing.
\]The map $\eta_{l,\bm a}$ is an isomorphism of algebras. Since $\mc B_{l}(\bm a)$ is finite-dimensional, by e.g. \cite[Lemma 3.9]{MTV09}, $\mc O_{l,\bm a}$ is a Frobenius algebra, so is 
$\mc B_{l}(\bm a)$. The map $\rho_{l,\bm a}$ is an isomorphism of vector spaces. Moverover, it follows from Theorem \ref{thm VS} and Lemma \ref{lem eta rho weyl} that $\rho_{l,\bm a}$ identifies the regular representation of $\mc O_{l,\bm a}$ with the $\mc B_{l}(\bm a)$-module $(V(\bm a))_{(n-l,l)}^\sing$. Therefore part (ii) is proved for the case of $V(\bs a)$.

\subsection{The third isomorphism}\label{sec third}
Recall from Section \ref{sec BA}, that without loss of generality, 
we can assume that $\la_2^{(s)}=0$, $s=1,\dots,k$. Rearrange the sequences 
$$
\{b_s,b_s-1,\dots,b_s-\la_1^{(s)}+1\},\qquad s=1,\dots,k,
$$
to a single sequence in decreasing order with respect to the real parts. Denote this new sequence by $\bm a=(a_1,\dots,a_n)$ (recall that $|\bla|=n$). Then by Lemma \ref{lem cyclic weyl}, $V(\bm a)$ is cyclic. 

\begin{lem}\label{lem quotient cyc}
If $L(\bla,\bm b)$ is cyclic, then there exists a surjective $\Yone$-module homomorphism from $V(\bm a)$ to $L(\bla,\bm b)$ which maps vacuum vector to vacuum vector.
\end{lem}
\begin{proof}
Rearrange the sequences
$$
\{b_s,b_s-1,\dots,b_s-\la_1^{(s)}+1\},\qquad s=1,\dots,k,
$$
to a single sequence in the same order displayed as $s$ runs from $1$ to $k$. Denote this new sequence by $\bm{a_0}$. Clearly, we have a surjective $\Yone$-module homomorphism $V(\bm{a_0})\twoheadrightarrow L(\bla,\bm b)$ which maps vacuum vector to vacuum vector. Therefore it suffices to show that there is a $\Yone$-module homomorphism $V(\bm a)\to V(\bm{a_0})$ which preserves the vacuum vector.

We have the $\Yone$-module homomorphism
\[
P\circ R(a-b):\C^{1|1}(a)\otimes \C^{1|1}(b)\to \C^{1|1}(b)\otimes \C^{1|1}(a).
\]Here $P$ is the graded flip operator, $R(a-b)$ is the R-matrix, see \eqref{eq R matrix}, and $a,b$ are complex numbers. Note that if $\Re a-\Re b\gge 0$, then $P\circ R(a-b)$ is well-defined and preserves the vacuum vector. Since $\Re a_i\gge \Re a_j$ for $1\lle i<j\lle n$, we obtain $\bm{a_0}$ by permuting elements in $\bm a$ via a sequence of simple reflections which move numbers with larger real parts through numbers with smaller real parts from left to right. Therefore a $\Yone$-module homomorphism $V(\bm a)\to V(\bm{a_0})$ which preserves the vacuum vector exists.

Since $L(\bla,\bs b)$ is cyclic, the map we constructed is surjective.
\end{proof}
By Proposition \ref{prop local weyl}, the surjective $\Yone$-module homomorphism $V(\bm a)\twoheadrightarrow L(\bla,\bm b)$ induces a surjective $\Yone$-module homomorphism $\mc V^\fkS\twoheadrightarrow L(\bla,\bm b)$. The second map then induces a projection of the Bethe algebras $\mc B_{l}(\bm z)\twoheadrightarrow \mc B_{l}(\bla,\bm b)$. We describe the kernel of this projection. We consider the corresponding ideal in the algebra $\mc O_l$.

Suppose $l\lle k-1$. Define the polynomial $h(x)$ by
\[
h(x)=\prod_{s=1}^k\prod_{i=1}^{\la_1^{(s)}-1}(x-b_s+i).
\]
If $L(\bla,\bs b)$ is irreducible then $h(x)$ is the greatest common divisor of
$\prod_{i=1}^n(x-a_i+1)$ and $\prod_{i=1}^n(x-a_i)$.
 
Divide the polynomial $g(x)$ in \eqref{eq poly f g} by $h(x)$ and let
\beq\label{eq coeff p}
p(x)=x^{k-l-1}+p_1x^{k-l-2}+\dots+p_{k-l-2}x+p_{k-l-1},
\eeq
\beq\label{eq coeff r}
r(x)=r_{1}x^{n-k-1}+r_2x^{n-k-2}+\dots+r_{n-k-1}x+r_{n-k}
\eeq
be the quotient and the remainder, respectively. Clearly, $p_i,r_j\in\mc O_l$. 

Denote by $I_{l,\bla,\bm b}^{\mc O}$ the ideal of $\mc O_l$ generated by $r_1,\dots,r_{n-k}$, $\vSi_n-a_1\cdots a_n$, and the coefficients of polynomial
\[
\prod_{s=1}^k(x-b_s+\la_1^{(s)})-\prod_{s=1}^k(x-b_s)-np(x)f(x).
\]
Let $\mc O_{l,\bla,\bm b}$ be the quotient algebra
\[
\mc O_{l,\bla,\bm b}=\mc O_l/I_{l,\bla,\bm b}^{\mc O}.
\]
Clearly, if $\mc O_{l,\bla,\bm b}$ is finite-dimensional, then it is a Frobenius algebra. 

Let $I_{l,\bla,\bm b}^{\mc B}$ be the image of $I_{l,\bla,\bm b}^{\mc O}$ under the isomorphism $\eta_{l}$.

\begin{lem}\label{lem ann ideal}
The ideal $I_{l,\bla,\bm b}^{\mc B}$ is contained in the kernel of the projection $\mc B_{l}(\bm z)\twoheadrightarrow \mc B_{l}(\bla,\bm b)$.
\end{lem}
\begin{proof}
We treat $\bm b=(b_1,\dots,b_k)$ as variables. Note that the elements of $I_{l,\bla,\bm b}^{\mc B}$ act on $(L(\bla,\bm b))^\sing_{(n-l,l)}$ polynomially in $\bm b$. Therefore it suffices to show it for generic $\bm b$. Let $\mathfrak f(x)$ be the image of $f(x)$ under $\eta_{l}$. The condition that $I_{l,\bla,\bm b}^{\mc B}$ vanishes is equivalent to the condition that $\varphi_{\bla,\bm b}(x)-\psi_{\bla,\bm b}(x)$ is divisible by $\mathfrak f(x)$.

By Theorems \ref{thm complete}, there exists a common eigenbasis of the transfer matrix $\cT(x)$ in $(L(\bla,\bm b))^\sing_{(n-l,l)}$ for generic $\bm b$. Let $\bm{\omega_1}=(\omega_1,\dots,\omega_1)$ consisting of $n$ of $\omega_1=(1,0)$. Clearly, a solution of Bethe ansatz equation associated to $\bla,\bm b,\l$ is also a solution to Bethe ansatz equation associated to $\bm{\omega_1},\bm a,l$. Moreover, the expressions of corresponding on-shell Bethe vectors differ by a scalar multiple (with different vacuum vectors). By Lemma \ref{lem quotient cyc} and Theorem \ref{thm gl11 eigenvalue in y}, $\varphi_{\bla,\bm b}(x)-\psi_{\bla,\bm b}(x)$ is divisible by $\mathfrak f(x)$ for generic $\bm b$ since the eigenvalue of $\mathfrak f(x)$ corresponding to $y_{\bm t}(x)$ in \eqref{eq gl11 eigenvalue in y}. Therefore $I_{l,\bla,\bm b}^{\mc B}$ vanishes for generic $\bm b$, completing the proof.
\end{proof}

Therefore, we have the epimorphism
\beq\label{eq epi new}
\mc O_{l,\bla,\bm b}\cong \mc B_{l}(\bm z)/I_{l,\bla,\bm b}^{\mc B}\twoheadrightarrow \mc B_{l}(\bla,\bm b).
\eeq
We claim that the surjection in \eqref{eq epi new} is an isomorphism by checking $\dim \mc O_{l,\bla,\bm b}=\dim \mc B_{l}(\bla,\bm b)$.

\begin{lem}\label{lem dimen count}
We have $\dim \mc O_{l,\bla,\bm b}=\displaystyle\binom{k-1}{l}$.
\end{lem}
\begin{proof}
Note that $\C[p_1,\dots,p_{k-l-1},r_1,\dots,r_{n-k}]\cong \C[g_1,\dots,g_{n-l-1}]$, where $p_i$ and $r_j$ are defined in \eqref{eq coeff p} and \eqref{eq coeff r}. It is not hard to check that
\beq\label{eq iso algebra}
\mc O_{l,\bla,\bm b}\cong \C[f_1,\dots,f_l,p_1,\dots,p_{k-l-1}]/\tilde I_{l,\bla,\bm b}^{\mc O},
\eeq
where $\tilde I_{l,\bla,\bm b}^{\mc O}$ is the ideal of $\C[f_1,\dots,f_l,p_1,\dots,p_{k-l-1}]$ generated by the coefficients of the polynomial $\prod_{s=1}^k(x-b_s+\la_1^{(s)})-\prod_{s=1}^k(x-b_s)-np(x)f(x)$.

Introduce new variables $\bm w=(w_1,\dots,w_{k-1})$ such that
\[
f(x)=\prod_{i=1}^{l}(x-w_i),\quad p(x)=\prod_{i=1}^{k-l-1}(x-w_{l+i}).
\]Let $\bm \varepsilon=(\varepsilon_1,\dots,\varepsilon_{k-1})$ be complex numbers such that
\[
\prod_{s=1}^k(x-b_s+\la_1^{(s)})-\prod_{s=1}^k(x-b_s)=n\Big(x^{k-1}+\sum_{i=1}^{k-1}(-1)^i\varepsilon_ix^{k-1-i}\Big).
\]
Then 
\beq\label{eq iso algebra new}
\C[f_1,\dots,f_l,p_1,\dots,p_{k-l-1}]/\tilde I_{l,\bla,\bm b}^{\mc O}\cong \C[w_1,\dots,w_{k-1}]^{\fkS_{l}\times \fkS_{k-l-1}}/\langle \sigma_i(\bm w)-\varepsilon_i\rangle_{i=1,\dots,k-1}.
\eeq
The lemma follows from the fact that $\C[w_1,\dots,w_{k-1}]^{\fkS_{l}\times \fkS_{k-l-1}}$ is a free module over $\C[w_1,\dots,w_{k-1}]^\fkS$ of rank $\binom{k-1}{l}$.\end{proof}

Note that we have the projection $(\mc V^\fkS)_{(n-l,l)}^\sing \twoheadrightarrow (L(\bla,\bm b))_{(n-l,l)}^\sing$. Since by Theorem \ref{thm VS} the Bethe algebra $\mc B_l(\bm z)$ acts on $(\mc V^\fkS)_{(n-l,l)}^\sing$ cyclically, the Bethe algebra $\mc B_l(\bla,\bm b)$ acts on $(L(\bla,\bm b))_{(n-l,l)}^\sing$ cyclically as well. Therefore we have
\[
\dim \mc B_l(\bla,\bm b)=\dim (L(\bla,\bm b))_{(n-l,l)}^\sing=\binom{k-1}{l}.
\]

\begin{proof}[Proof of Theorem \ref{thm tensor irr}]
Part (i) follows from Lemma \ref{lem dimen count} and \eqref{eq epi new}, \eqref{eq iso algebra}, \eqref{eq iso algebra new}. Clearly, we have $\mc B_l(\bla,\bm b)\cong \mc O_{l,\bla,\bm b}$ is a Frobenius algebra. Moreover, the map $\rho_l$ from Theorem \ref{thm VS} induces a map
\[
\rho_{l,\bla,\bm b}:\mc O_{l,\bla,\bm b}\to (L(\bla,\bm b))_{(n-l,l)}^\sing
\]
which identifies the regular representation of $\mc O_{l,\bla,\bm b}$ with the $\mc B_{l}(\bla,\bm b)$-module $(L(\bla,\bm b))_{(n-l,l)}^\sing$. Therefore part (ii) is proved.

Since $\mc B_{l}(\bla,\bm b)$ is a Frobenius algebra, the regular and coregular representations of $\mc B_{l}(\bla,\bm b)$ are isomorphic to each other. Parts (iii)--(vi) follow from the general facts about the coregular representations, see e.g. \cite[Section 3.3]{MTV09}. 

Due to part (iv), it suffices to consider the algebraic multiplicity of every eigenvalue. It is well known that roots of a polynomial depend continuously on its coefficients. Hence the eigenvalues of $\mc T(x)$ depend continuously on $\bm b$. Part (vii) follows from deformation argument and Theorem \ref{thm complete}.
\end{proof}

\begin{proof}[Proof of Theorem \ref{thm tensor irr q}]
It is parallel to that of Theorem \ref{thm tensor irr} with the following minor modification.

The degree of $g(x)$ is $l$ instead of $l-1$. Divide the polynomial $g(x)$ in \eqref{eq poly f g} by $h(x)$ and let
\[
p(x)=x^{k-l}+p_1x^{k-l-1}+\dots+p_{k-l-1}x+p_{k-l},
\] 
\[
r(x)=r_{1}x^{n-k-1}+r_2x^{n-k-2}+\dots+r_{n-k-1}x+r_{n-k}
\]be the quotient and remainder, respectively. Clearly, $p_i,r_j\in\mc O^Q_l$. 

Denote by $I_{l,\bla,\bm b}^{Q,\mc O}$ the ideal of $\mc O^Q_l$ generated by $r_1,\dots,r_{n-k}$ and the coefficients of polynomial
\[
q_1\prod_{s=1}^k(x-b_s+\la_1^{(s)})-q_2\prod_{s=1}^k(x-b_s-\la_2^{(s)})-(q_1-q_2)p(x)f(x).
\]
Let $\mc O^Q_{l,\bla,\bm b}$ be the quotient algebra
\[
\mc O^Q_{l,\bla,\bm b}=\mc O^Q_l/I_{l,\bla,\bm b}^{Q,\mc O}.
\]
The rest of the proof is similar, we only note that $\dim(\mc O^Q_{l,\bla,\bm b})= {k \choose l}$.
\end{proof}

\begin{proof}[Proof of Theorem \ref{thm completeness general}]
We give a proof for the case $q_1=q_2$. The case $q_1\ne q_2$ is similar.

By part (v) of Theorem \ref{thm tensor irr}, all eigenvalues of $\cT(x)$ have the form \eqref{eq gl11 eigenvalue in y} with a monic divisor $y_{\bm t}$ of $\varphi_{\bla,\bm b}(x)-\psi_{\bla,\bm b}(x)$. Moreover, it follows from part (iv) of Theorem \ref{thm tensor irr} that all eigenspaces of $\cT(x)$ have dimension one. By Lemma \ref{lem nonzero}, all on-shell Bethe vectors are non-zero and the theorem follows.
\end{proof}

\section{Higher transfer matrices}\label{sec higher}
\subsection{Higher transfer matrices}
We have the standard action of symmetric group $\fkS_{m}$ on the space $(\C^{1|1})^{\otimes m}$ where $s_i$ acts as the graded flip operator $P^{(i,i+1)}$. We denote by $A_m$ and $H_m$ the images of the normalized anti-symmetrizer and symmetrizer,
\[
A_m=\frac{1}{m!}\sum_{\sigma\in\fkS_m}\mathrm{sign}(\sigma)\cdot \sigma,\qquad H_m=\frac{1}{m!}\sum_{\sigma\in\fkS_m} \sigma.
\]

Define the \emph{$m$-th transfer matrix associated to the diagonal matrix}  $Q=\mathrm{diag}(q_1,q_2)$ by
\beq\label{eq m-transfer}
\fkT^Q_m(x)=\str(A_mQ^{(1)}T^{(1,m+1)}(x)Q^{(2)}T^{(2,m+1)}(x-1)\cdots Q^{(m)}T^{(m,m+1)}(x-m+1)).
\eeq
Here the supertrace is taken over all copies of $\End(\C^{1|1})$. Clearly, $\fkT_1^Q(x)=\cT_Q(x)$. 

Expand $\fkT_m^Q(x)$ as a power series in $x^{-1}$,
\[
\fkT_m^Q(x)=\sum_{s=0}^\infty B^Q_{m,s}x^{-s}.
\]
We denote the unital subalgebra of $\Yone$ generated by elements $B^Q_{m,s}$ for all $m\gge 1$ and $s\gge 0$ by $\mathfrak B^Q$. When $q_1=q_2$, we simple write $\mathfrak B$ for $\mathfrak B^Q$. The subalgebra $\mathfrak B^Q$ does not change if $q_1,q_2$ are multiplied by the same non-zero number. Therefore if $q_1=q_2$, we assume further that $q_1=1$.

The following statements are standard.

\begin{prop}
The algebra $\mathfrak B^Q$ is commutative. If $q_1\ne q_2$, then $\mathfrak B^Q$ contains the algebra $\mathrm{U}(\h)$ and hence commutes with $\mathrm{U}(\h)$. If $q_1=q_2$, the algebra $\mathfrak B^Q$ commutes with the algebra $\Uone$.
\end{prop}
\begin{proof}
The first statement follows from the RTT relation \eqref{eq RTT}, c.f. \cite[Proposition 4.5]{MTV}. The second is a corollary of the formulas
\[
B_{1,1}^Q=q_1T_{11}^{(1)}-q_2T_{22}^{(1)},\qquad B_{2,1}^Q=q_1T_{11}^{(1)}+(q_1-2q_2)T_{22}^{(1)}.
\]
The last statement is obtained from \eqref{eq com gll}, c.f. \cite[Proposition 4.7]{MTV}.
\end{proof}

As a subalgebra of $\Yone$, the algebra $\mathfrak B^Q$ acts naturally on any $\Yone$-module $M$. Since $\mc B^Q$ commutes with $\mathrm{U}(\h)$, it preserves the weight spaces $(M)_{\la}$. Moreover, $\mc B$ preserves the weight singular spaces $(M)_{\la}^\sing$. 

\begin{prop}
The algebra $\mc B^Q$ is stable under the anti-automorphism $\iota$ in \eqref{eq anti-in}, $\iota(\fkT_m^Q(x))=\fkT_m^Q(x)$.
\end{prop}
\begin{proof}
Note that $\iota(T(x))=(T(x))^t$, $Q^t=Q$, and $(A_m)^t=A_m$, the proof is parallel to that of \cite[Proposition 4.11]{MTV}. Here $t$ is the supertranspose and the supertransposition $(A_m)^t$ is taken over all copies of $\End(\C^{1|1})$ in $A_m$.
\end{proof}

\subsection{Berezinian and rational difference operator}	
Define the rational difference operator $\mc D^Q(x,\tau)$,
\[
\mc D^Q(x,\tau)=\mathrm{Ber}(1-Z^Q(x,\tau)),
\]
where as before $Z^Q(x,\tau)=T^t(x)Q\tau$.

Applying the supertransposition to all copies of $\End(\C^{1|1})$ and using cyclic property of supertrace, see \eqref{eq transpose cyclic},  one has
$$
\fkT^Q_m(x)=\str(A_m(T^t)^{(1,m+1)}(x)Q^{(1)}(T^t)^{(2,m+1)}(x-1)Q^{(2)}\cdots (T^t)^{(m,m+1)}(x-m+1)Q^{(m)}).
$$
Therefore by \cite[Theorem 2.13]{MR14} we have
\beq\label{eq diff trans}
\mc D^Q(x,\tau)=\sum_{m=0}^\infty (-1)^m \fkT_m^Q(x)\tau^m.
\eeq
By \eqref{ber}, we obtain
$$
\mc D^Q(x,\tau)=(1-q_1T_{11}(x)\tau+q_1T_{12}(x)\tau(1-q_2T_{22}(x)\tau)^{-1}q_2T_{21}(x)\tau)(1-q_2T_{22}(x)\tau)^{-1}.
$$
Expand $(1-q_2T_{22}(x)\tau)^{-1}$ as a power series in $\tau$,
\[
(1-q_2T_{22}(x)\tau)^{-1}=\sum_{m=0}^{\infty}(q_2T_{22}(x)\tau)^m=\sum_{m=0}^{\infty}q_2^m\prod_{i=1}^mT_{22}(x-i+1)\tau^m,
\] 
and compare to \eqref{eq diff trans}. It gives $\fkT^Q_{1}(x)=\mc T_Q(x)=q_1T_{11}(x)-q_2T_{22}(x)$ and for  $m\gge 2$,
\begin{align}
\widetilde \fkT^Q_m(x):= (-1)^m q_2^{1-m}\fkT^Q_m(x)=&\ -(q_1T_{11}(x)-q_2T_{22}(x))\prod_{i=1}^{m-1}T_{22}(x-i)+\nonumber \\&\sum_{s=1}^{m-1}q_1T_{12}(x)\Big(\prod_{i=1}^{s-1}T_{22}(x-i)\Big)T_{21}(x-s)\prod_{j=s+1}^{m-1}T_{22}(x-j).\label{eq expansion higher}
\end{align}

\begin{rem}
The expansion \eqref{eq expansion higher} (and other variations) of the higher transfer matrices $\fkT^Q_m(x)$ can also be obtained from \cite[Proposition 2.3, Remark 2.4]{MR14}.
\end{rem}

Let, as in Section \ref{sec BA}, $\bm b=(b_1,\dots,b_k)$ be a sequence of complex numbers, $\bla=(\la^{(1)},\dots,\la^{(k)})$ a sequence of $\gl_{1|1}$-weights. Let $\bm t=(t_1,\dots,t_l)$ be a solution of the Bethe ansatz equation \eqref{eq gl11 BAE}. Define two rational functions
\beq\label{eq highest l-weight}
\zeta_{1}(x)=\prod_{s=1}^k\frac{x-b_s+\la_1^{(s)}}{x-b_s},\qquad \zeta_{2}(x)=\prod_{s=1}^k\frac{x-b_s-\la_2^{(s)}}{x-b_s}.
\eeq
We also use the following notation,
\[
f^{[i]}:=\tau^i(f)=f(x-i)
\]
for any function $f$ in $x$.

Let $y=(x-t_1)\cdots(x-t_l)$. In \cite{HLM}, we associate a rational difference operator $\mc D_{\bm t,\bla,\bm b}^Q(x,\tau)$ (or $\mc D_{y,\bla,\bm b}^Q(x,\tau)$) to each solution $\bm t$ of the Bethe ansatz equation,
\beq\label{eq diff oper y}
\mc D_{\bm t,\bla,\bm b}^Q(x,\tau)=\mc D_{y,\bla,\bm b}^Q(x,\tau)=\Big(1-q_1\zeta_1\frac{y^{[1]}}{y}\tau\Big)\Big(1-q_2\zeta_2\frac{y^{[1]}}{y}\tau\Big)^{-1}.
\eeq
The operator $\mc D_{\bm t,\bla,\bm b}^Q(x,\tau)$ describes the eigenvalues of the algebra $\mathfrak B^Q$ acting on the corresponding on-shell Bethe vector $\widehat\bB_l(\bm t)$.

\begin{thm}\label{thm eigenvalue oper}
Assume that $t_i\ne t_j$ for $i\ne j$. We have 
\[
\mc D^Q(x,\tau)\widehat\bB_l(\bm t)=\mc D_{\bm t,\bla,\bm b}^Q(x,\tau)\widehat\bB_l(\bm t).
\]
\end{thm}
We give the proof of this theorem in the next section. 

\subsection{Proof of Theorem \ref{thm eigenvalue oper}}
Consider the expansion of the rational difference operator $\mc D_{\bm t,\bla,\bm b}^Q(x,\tau)$,
\[
\mc D_{\bm t,\bla,\bm b}^Q(x,\tau)=1-\sum_{m=1}^{\infty}q_2^{m-1}(q_1\zeta_1-q_2\zeta_2)\frac{y^{[m]}}{y}\Big(\prod_{i=1}^{m-1}\zeta_2^{[i]}\Big)\tau^m.
\]
Therefore, it suffices to show that
\beq\label{eq to show}
\widetilde \fkT_m^Q(x)\,\widehat\bB_{l}(\bm t)=-(q_1\zeta_1-q_2\zeta_2)\frac{y^{[m]}}{y}\Big(\prod_{i=1}^{m-1}\zeta_2^{[i]}\Big)\widehat\bB_{l}(\bm t).
\eeq
We split the proof into three steps. It would be convenient to work with the unrenormalized Bethe vector $\bB_l(\bm t)$.

\subsubsection{Actions of $T_{ij}(x)$ on Bethe vectors}
We prepare several lemmas for the proof. Following \cite{HLPRS16}, we set
\[
\mu(x_1,x_2)=\frac{x_1-x_2+1}{x_1-x_2},\quad \nu(x_1,x_2)=\frac{1}{x_1-x_2},\quad \kappa(x_1,x_2)=x_1-x_2+1.
\] Note that we have $\nu(x_1,x_2)\kappa(x_1,x_2)=\mu(x_1,x_2)$.

For a sequence of complex numbers $\bm t=(t_1,\dots,t_l)$ define sequences of complex numbers $\bti$ and $\btij$ by
\[
\bti=(t_1,\dots,t_{i-1},t_{i+1},\dots,t_l),\quad 1\lle i\lle l,
\]
\[
\btij=(t_1,\dots,t_{i-1},t_{i+1},\dots,t_{j-1},t_{j+1},\dots,t_l), \quad 1\lle i<j\lle l.
\]
We use the shorthand notation as follows. Let $\bm u=(u_1,\dots,u_r)$ and $\bm w=(w_1,\dots,w_s)$ be sequences of complex numbers. Set
\[
\mu(x,\bm u)=\prod_{i=1}^r\mu(x,u_i),\quad \mu(\bm u,\bm w)=\prod_{i=1}^r\prod_{j=1}^s\mu(u_i,w_j). 
\]The same convention also applies to functions $\nu(x_1,x_2)$, $\kappa(x_1,x_2)$, and currents $T_{ii}(x)$, etc. 

By $\bB_{l}(\bti,z)$ and $\bB_{l-1}(\btij,z)$, we mean the off-shell Bethe vectors \eqref{eq bv gl11} associated to the sequences $\bti \sqcup \{z\}$ and $\btij\sqcup\{z\}$, respectively.

\begin{lem}[\cite{HLPRS16}] \label{lem comm mixed}
We have $T_{12}(z)\bB_l(\bm t)=\zeta_1(z)\kappa(\bm t,z)\bB_{l+1}(\bm t,z)$,
\begin{align*}
T_{21}(z)\bB_l(\bm t)=&\ \zeta_1(z)\sum_{i=1}^l\mu(\bti,z)\nu(\bti,t_i)\nu(t_i,z)\Big(\frac{\zeta_2(z)}{\zeta_1(z)}-\frac{\zeta_2(t_i)}{\zeta_1(t_i)}\Big)\bB_{l-1}(\bti)
\\
+\,\zeta_1(z)\sum_{1\lle i<j\lle l}& \nu(z,t_i)\nu(z,t_j)\nu(\btij,t_i)\nu(\btij,t_j)\kappa(\btij,z)\nu(t_j,t_i)\Big(\frac{\zeta_2(t_i)}{\zeta_1(t_i)}-\frac{\zeta_2(t_j)}{\zeta_1(t_j)}\Big)\bB_{l-1}(\btij,z),\\
T_{11}(z)\bB_l(\bm t)=&\ \zeta_1(z)\mu(\bm t,z)\bB_l(\bm t)+\zeta_1(z)\sum_{i=1}^l \kappa(\bti,z)\nu(\bti,t_i)\nu(z,t_i)\bB_l(\bti, z).\qedd
\end{align*}
\end{lem}

For $b\in\Z_{>0}$, let $\bm{z_b}=\{z,z-1,\dots,z-b+1\}$ and $\bm{z_b}^\circ=\{z-1,\dots,z-b+1\}$.
\begin{lem}[\cite{HLPRS16}]\label{lem commu T22} 
We have
\begin{align*}
T_{22}(\bm{z_b})\bB_l(\bm t)=\,&\zeta_2(\bm{z_b})\mu(\bm t,\bm{z_b})\bB_l(\bm t)\\ +\,&\zeta_2(\bm{z_b}^\circ)\zeta_1(z)\mu(\bti,\bm{z_b}^\circ)\sum_{i=1}^l\frac{\zeta_2(t_i)}{\zeta_1(t_i)} \mu(z,\bm{z_b}^\circ)\kappa(\bti,z)\nu(\bti,t_i)\nu(z,t_i)\bB_l(\bti,z).
\end{align*}
\end{lem}
\begin{proof}
This is a particular case of \cite[equation (3.6)]{HLPRS16}. A number of simplifications occur for the special choice of $\bm{z_b}$.
\end{proof}

\subsubsection{Strategy of computation}\label{sec strategy}
We aim at \eqref{eq to show}. Let $\bm t$ be a solution of Bethe ansatz equation. Let $\bals=\{x-1,\dots,x-s+1\}$ and $\bbes=\{x-s-1,\dots,x-m+1\}$, $s=1,\dots,m-1$, where $\bm{\alpha_1}=\bm{\beta_{m-1}}=\varnothing $. Then
\[
\widetilde\fkT_m^Q(x)=-(q_1T_{11}(x)-q_2T_{22}(x))T_{22}(\bm{\alpha_{m}})+\sum_{s=1}^{m-1}q_1 T_{12}(x)T_{22}(\bals)T_{21}(x-s)T_{22}(\bbes).
\]

We consider the action of each term in the summation above on the on-shell Bethe vector $\bB_l(\bm t)$.

Consider the action of $(q_1T_{11}(x)-q_2T_{22}(x))T_{22}(\bm{\alpha_{m}})$ on $\bB_l(\bm t)$. By Lemma \ref{lem commu T22}, $T_{22}(\bm{\alpha_{m}}) \bB_l(\bm t)$ is a linear combination of Bethe vectors $\bB_l(\bm t)$ and $\bB_l(\bti,x-1)$. Note that $\bB_l(\bm t)$ is an eigenvector of $q_1T_{11}(x)-q_2T_{22}(x)$ and $$T_{ii}(x)T_{12}(x-1)=T_{12}(x)T_{ii}(x-1), \qquad i=1,2,$$ it follows that $(q_1T_{11}(x)-q_2T_{22}(x))T_{22}(\bm{\alpha_{m}})\bB_l(\bm t)$ is a linear combination of $\bB_l(\bm t)$ and $\bB_l(\bti,x)$.

Consider the vector 
\beq\label{eq:bvaction}
T_{12}(x)T_{22}(\bals)T_{21}(x-s)T_{22}(\bbes)\bB_l(\bm t).
\eeq
Again by Lemma \ref{lem commu T22}, $T_{22}(\bbes) \bB_l(\bm t)$ is a linear combination of Bethe vectors $\bB_l(\bm t)$ and $\bB_l(\bti,x-s-1)$. After the action of $T_{21}(x-s)$, it follows from Lemma \ref{lem comm mixed} that we get Bethe vectors $\bB_{l-1}(\bti)$ and $\bB_{l-1}(\btij,x-s)$. Then the action of $T_{22}(\bals)$ on $\bB_{l-1}(\bti)$ gives $\bB_{l-1}(\bti)$ and $\bB_{l-1}(\btij,x-1)$ while the action of $T_{22}(\bals)$ on $\bB_{l-1}(\btij,x-s)$ gives $\bB_{l-1}(\bti,x-1)$. Note that $$T_{12}(x)T_{12}(x-1)=0,$$ hence the final result only involves $\bB_{l}(\bti,x)$ for $i=1,\dots,l$. The vectors we are obtaining in \eqref{eq:bvaction} are described by the following picture.

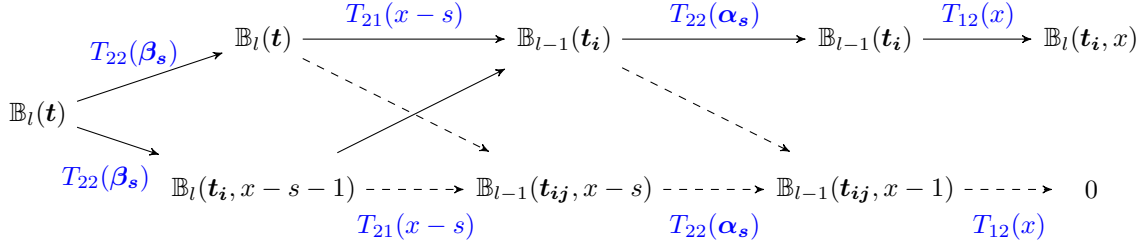
\begin{center}
\begin{tikzpicture}[->,>=stealth',shorten >=1pt,auto,node distance=2.8cm]
\tikzstyle{every state}=[rectangle,draw=none,text=black]

\node[state]         (S) at (-11, 1.6)              {$\bB_l(\bm t)$};
\node[state]         (xin1) at (-8, 2.6)           {$\bB_l(\bm t)$};
\node[state]         (xin2) at (-4, 2.6)        {$\bB_{l-1}(\bti)$};
\node[state]         (xin3) at (0, 2.6)        {$\bB_{l-1}(\bti)$};
\node[state]         (xout1) at (0,0.6 )          {$\bB_{l-1}(\btij,x-1)$};
\node[state]         (xout2) at (-4, 0.6)        {$\bB_{l-1}(\btij,x-s)$};
\node[state]         (xout6) at (-8, 0.6)        {$\bB_l(\bti,x-s-1)$};
\node[state]         (xout8) at (3, 0.6)        {$0$};
\node[state]         (DC) at (3, 2.6)           {$\bB_l(\bti,x)$};

\path
(S) edge node[xshift=0.5cm,yshift=-0cm]{$\color{blue}T_{22}(\bbes)$}   (xin1)
(S) edge node[xshift=-0.9cm,yshift=-0.8cm]{$\color{blue}T_{22}(\bbes)$}  (xout6)
(xin1) edge node[xshift=0cm,yshift=-0cm]{$\color{blue}T_{21}(x-s)$}   (xin2)
(xin1) edge[dashed] node[xshift=0cm,yshift=-0cm]{}   (xout2)
(xin2) edge node[xshift=0cm,yshift=-0cm]{$\color{blue}T_{22}(\bals)$}  (xin3)
(xin2) edge[dashed]  node[xshift=0.3cm,yshift=-0.2cm]{} (xout1)
(xin3) edge node[xshift=0cm,yshift=-0cm]{$\color{blue}T_{12}(x)$}  (DC)		
(xout6) edge[dashed] node[xshift=0cm,yshift=-0.8cm]{$\color{blue}T_{21}(x-s)$} (xout2)
(xout6) edge node[xshift=-0.1cm,yshift=-0.2cm]{} (xin2)
(xout1) edge[dashed] node[xshift=0cm,yshift=-0.8cm]{$\color{blue}T_{12}(x)$} (xout8)
(xout2) edge[dashed] node[xshift=0cm,yshift=-0.8cm]{$\color{blue}T_{22}(\bals)$} (xout1);
\end{tikzpicture}
\end{center}

More precisely, the picture describes the result of action of operators to linear combinations of various Bethe vectors. We apply the same operator to all vectors in the same column. This operator is indicated on the top of the solid line in the first row and also at the bottom of the second row. Then the arrows show which vectors are obtained in each case. Dashed arrows correspond to terms which eventually become zero. Solid lines correspond to terms which have a non-trivial contribution.
For example, $T_{21}(x-s) \bB_l(\bm t)$ is a linear combination of Bethe vectors $\bB_{l-1}(\bm t_i)$ and $\bB_{l-1}(\btij,x-s)$ with $i=1,\dots,l$, $j=i+1,\dots,l$. The latter terms will be all annihilated by further applications of $T_{22}(\bs \alpha_{\bs s})$ and $T_{12}(x)$.

\subsubsection{End of proof}
Note that $q_1\zeta_1(t_i)=q_2\zeta_2(t_i)$ and $\nu(x_1,x_2)\kappa(x_1,x_2)=\mu(x_1,x_2)$. By Lemmas \ref{lem comm mixed} and \ref{lem commu T22}, following the way described in Section \ref{sec strategy}, we obtain that $T_{12}(x)T_{22}(\bals)T_{21}(x-s)T_{22}(\bbes)\bB_l(\bm t)$ is equal to
\begin{align}
\sum_{i=1}^l&\zeta_1(x)\zeta_2(\bm{\alpha_m})\nu(\bti,t_i)\kappa(\bti,x)\mu(\bti,\bm{\alpha_m}) \left(\mu(t_i,\bbes)\nu(x-s,t_i)\left(\frac{q_1\zeta_1(x-s)}{q_2\zeta_2(x-s)}-1\right)\right. \label{eq bethe 1}
\\  +&(1-\delta_{s,m-1})\mu(x-s-1,\bm{\beta_{s+1}})\nu(x-s-1,t_i)\left.\left(\frac{q_1\zeta_1(x-s)}{q_2\zeta_2(x-s)}-\frac{q_1\zeta_1(x-s-1)}{q_2\zeta_2(x-s-1)}\right)\right)\bB_l(\bti,x).\nonumber
\end{align}
Similarly, 
\begin{align}
&(q_1T_{11}(x)-q_2T_{22}(x))T_{22}(\bm{\alpha_{m}})\bB_l(\bm t)=(q_1\zeta_1(x)-q_2\zeta_2(x))\zeta_2(\bm{\alpha_m})\mu(\bm t,x)\mu(\bm t,\bm{\alpha_m})\bB_l(\bm t) \label{eq bethe 2}\\
+&\sum_{i=1}^lq_1\zeta_1(x)\zeta_2(\bm{\alpha_m})\nu(\bti,t_i)\kappa(\bti,x)\mu(\bti,\bm{\alpha_m})\mu(x-1,\bm{\alpha_m}^\circ)\nu(x-1,t_i)\left(\frac{q_1\zeta_1(x-1)}{q_2\zeta_2(x-1)}-1\right)\bB_l(\bti,x).\nonumber
\end{align}

Therefore, $\widetilde\fkT^Q_{m}(x)\bB_l(\bm t)$ is equal to the sum of \eqref{eq bethe 2} and the summation of \eqref{eq bethe 1} over $s=1,\dots,m-1$. For fixed $i\in \{1,\dots,l\}$ and $s\in \{1,\dots,m-1\}$, we combine all terms containing $\zeta_1(x)\zeta_2(\bm{\alpha_m})q_1\zeta_1(x-s)/(q_2\zeta_2(x-s))$ in \eqref{eq bethe 1}, \eqref{eq bethe 2} and consider the corresponding coefficient. To show \eqref{eq to show}, we first show that this coefficient vanishes. This follows from the following lemma.

\begin{lem}\label{lem raional 2}
We have $\mu(t_i,\bbes)\nu(x-s,t_i)+\mu(x-s-1,\bm{\beta_{s+1}})\nu(x-s-1,t_i)=\mu(x-s,\bbes)\nu(x-s,t_i)$, for $s=1,\dots,m-2$.\qed
\end{lem}
If we set $\mu(x-s-1,\bm{\beta_{s+1}})=0$, then the lemma also holds for $s=m-1$.

\medskip

We then combine the terms containing $\zeta_1(x)\zeta_2(\bm{\alpha_m})$ in this sum. The next lemma asserts this coefficient is equal to zero.
\begin{lem}\label{lem raional 1}
We have 
\[
\sum_{s=1}^{m-1}\mu(t_i,\bbes)\nu(x-s,t_i)=\mu(x-1,\bm{\alpha_m}^\circ)\nu(x-1,t_i).
\]
\end{lem}
\begin{proof}
We have
\begin{align*}
\sum_{s=1}^{m-1}\mu(t_i,\bbes)\nu(x-s,t_i)=&\sum_{s=1}^{m-1}\prod_{j=s+1}^{m-1}\frac{t_i-(x-j)+1}{t_i-(x-j)}\frac{1}{x-s-t_i}=\sum_{s=1}^{m-1}\frac{x-m-t_i}{(x-s-1-t_i)(x-s-t_i)}\\
=&(x-m-t_i)\sum_{s=1}^{m-1}\left(\frac{1}{x-s-1-t_i}-\frac{1}{x-s-t_i}\right)=\frac{m-1}{x-1-t_i}.
\end{align*}
Clearly, $\mu(x-1,\bm{\alpha_m}^\circ)=m-1$. Therefore the lemma follows.
\end{proof}

Thus we have
\[
\widetilde\fkT^Q_{m}(x)\,\mathbb B_{l}(\bm t)=(q_1\zeta_1(x)-q_2\zeta_2(x))\zeta_2(\bm{\alpha_m})\mu(\bm t,x)\mu(\bm t,\bm{\alpha_m})\bB_l(\bm t).
\]
Since $$\mu(\bm t,x)=\prod_{i=1}^l\frac{t_i-x+1}{t_i-x}=\frac{y(x-1)}{y(x)},$$
we have
\[
(q_1\zeta_1(x)-q_2\zeta_2(x))\zeta_2(\bm{\alpha_m})\mu(\bm t,x)\mu(\bm t,\bm{\alpha_m})=(q_1\zeta_1-q_2\zeta_2)\frac{y^{[m]}}{y}\prod_{i=1}^{m-1}\zeta_2^{[i]},
\]
which completes the proof of Theorem \ref{thm eigenvalue oper}.

\subsection{Relations between transfer matrices}
Similar to \eqref{eq m-transfer}, define
\[
\mathfrak H_m^Q(x)=\str(H_mQ^{(1)}T^{(1,m+1)}(x)Q^{(2)}T^{(2,m+1)}(x-1)\cdots Q^{(m)}T^{(m,m+1)}(x-m+1)).
\]
Using Theorem \ref{thm eigenvalue oper}, we are able to express $\fkT^Q_m(x)$ and $\mathfrak H_m^Q(x)$ in terms of the first transfer matrix $\cT_Q(x)$ and the center $\mathrm{Ber}^Q(x)$.

We start with the following technical lemma.
\begin{lem}\label{lem no killers}
No non-zero element in $\Yone$ acts by zero on $L(\bla,\bm b)$ for generic $\bla$ and $\bm b$.
\end{lem}

\begin{proof}
The lemma follows from the proof of \cite[Proposition 2.2]{Naz99}.
\end{proof}

\begin{thm}\label{thm transfer relations}
We have
\[
\fkT^Q_m(x)\prod_{i=1}^{m-1}(1-\mathrm{Ber}^Q(x-i))=\prod_{i=1}^m \cT_Q(x-i+1),
\]
\[
\mathfrak H_m^Q(x)\prod_{i=1}^{m-1}(\mathrm{Ber}^Q(x-i)-1)=\prod_{i=1}^m \cT_Q(x-i+1)\prod_{i=1}^{m-1}\mathrm{Ber}^Q(x-i).
\]
\end{thm}
\begin{proof}
We prove it for the case $q_1\ne q_2$. The case $q_1=q_2$ is similar.

Note that for an on-shell Bethe vector $\widehat\bB_l(\bm t)$, where $\bm t=(t_1,\dots,t_l)$ with $t_i\ne t_j$ for $i\ne j$, by \eqref{eq to show} and \eqref{eq gl11 eigenvalue in y}, we have
\[
\fkT^Q_m(x)\prod_{i=1}^{m-1}(1-\mathrm{Ber}^Q(x-i))\widehat\bB_l(\bm t)=\prod_{i=1}^m \cT_Q(x-i+1)\widehat\bB_l(\bm t).
\]
By Theorem \ref{thm complete}, the transfer matrix $\mc T_Q(x)$ is diagonalizable and the Bethe ansatz is complete for generic $\bla$ and $\bm b$, namely there exists a basis of $L(\bla,\bm b)$ consisting of on-shell Bethe vectors. Therefore, the coefficients of the formal series $$\fkT^Q_m(x)\prod_{i=1}^{m-1}(1-\mathrm{Ber}^Q(x-i))-\prod_{i=1}^m \cT_Q(x-i+1)$$
act by zero on $\Yone$-modules $L(\bla,\bm b)$ for generic $\bla$ and $\bm b$. The first equality of the theorem follows from Lemma \ref{lem no killers}.

By \cite[Theorem 2.13]{MR14}, one has
\[
(\mc D^Q(x,\tau))^{-1}=(\mathrm{Ber}^Q(x))^{-1}=\sum_{m=0}^\infty \mathfrak H_m^Q(x)\tau^m.
\]Therefore, the second equality is proved similarly.
\end{proof}

\begin{rem}
Equation \eqref{eq tensor n} in Example \ref{eg vect rep} is the second equality of Theorem \ref{thm transfer relations} on the representation level. We explain it in more detail for the case $q_1=q_2$. Following e.g. \cite[Section 3.1]{FR99}, let $\mathfrak R$ be the universal R-matrix in the $\Yone$ Yangian double $\mathrm{Y}^*(\mathfrak{gl}_{1|1})\widehat\otimes \Yone$. For a finite-dimensional representation $(V,\rho_V)$ of the dual Yangian $\mathrm{Y}^*(\mathfrak{gl}_{1|1})$, let
\[
\mathscr T_V(x)=\mathrm{str}\big((\rho_{V(x)}\otimes \mathrm{id})\mathfrak R\big)\in \Yone[[x^{-1}]].
\]Similar to \cite[Lemma 2]{FR99}, one has $\mathscr T_{V\otimes W}(x)=\mathscr T_V(x)\mathscr T_W(x)$. Moreover $\mathscr T_W(x)=\mathscr T_V(x)+\mathscr T_U(x)$ for a short exact sequence $V\hookrightarrow W \twoheadrightarrow U$. We expect that after proper rescaling of the universal R-matrix $\mathfrak R$ one has
\[
\mathscr T_{L_{\omega_1}(0)}(x)=\mc T(x),\quad \mathscr T_{L_{m\omega_1}(0)}(x)=\mathfrak H_m(x),
\]
and $\mathscr T_{\C_{\bar 1,\xi_m}}(x)$ is a certain rational function in $\mathrm{Ber}(x)$. Note that the modules here should be replaced with corresponding $\mathrm{Y}^*(\mathfrak{gl}_{1|1})$-modules. The first equality of Theorem \ref{thm transfer relations} can be understood similarly.\qed
\end{rem}
\begin{rem}\label{two Bethe algebras}
Recall that we call $\mathcal B^Q$ the Bethe algebra.
Often, it is the algebra $\mathfrak B^Q$  which is named the Bethe algebra. Due to Theorem \ref{thm transfer relations}, the images of $\mathfrak B^Q$ and $\mathcal B^Q$ acting on the modules $\mc V^{\fkS}$, $L(\bla,\bm b)$ with $|\bla|\ne 0$ coincide.\qed
\end{rem}

\begin{cor}\label{universal oper}
We have
\begin{align*}
\mathcal D^Q(x,\tau)&=\Big(1-\frac{\mathrm{Ber}^Q(x)\mc T_Q(x)}{\mathrm{Ber}^Q(x)-1}\, \tau\Big)\Big(1-\frac{\mc T_Q(x)}{\mathrm{Ber}^Q(x)-1}\, \tau\Big)^{-1}\\&=\big(1-\mathrm{Ber}^Q(x+1)+\mathrm{Ber}^Q(x)\mc T_Q(x)\, \tau\big)\big(1-\mathrm{Ber}^Q(x+1)+\mc T_Q(x)\, \tau\big)^{-1}.
\end{align*}
\end{cor}
\begin{proof}
The statement follows from Theorem \ref{thm transfer relations} by a direct computation.
\end{proof}
\begin{rem}
Note that
\[
\frac{\mathrm{Ber}^Q(x)\mc T_Q(x)}{\mathrm{Ber}^Q(x)-1}\, \bB_l(\bm t)=q_1\zeta_1\frac{y^{[1]}}{y}\, \bB_l(\bm t),\quad \frac{\mc T_Q(x)}{\mathrm{Ber}^Q(x)-1}\, \bB_l(\bm t)=q_2\zeta_2\frac{y^{[1]}}{y}\, \bB_l(\bm t),
\]
see \eqref{eq highest l-weight} and \eqref{eq diff oper y}.\qed
\end{rem}

\end{document}